\documentclass[11pt]{article}
\usepackage{fullpage}
\usepackage{amsmath,bm}
\usepackage{moresize}
\usepackage{amssymb}
\usepackage{tikz-cd}
\usepackage{graphicx}
\usepackage{parskip}
\usepackage{enumitem}
\usepackage{amsthm}
\usepackage{bbm}
\theoremstyle{definition}
\usepackage{listings}
\usepackage{color}
\usepackage{cancel}
\usepackage{framed}
\usepackage{relsize}
\usepackage{yfonts}
\usepackage{anyfontsize}
\usepackage{t1enc}
\usepackage[ruled,vlined]{algorithm2e}
\usepackage{algpseudocode}
\allowdisplaybreaks
\usepackage[utf8]{inputenc}
\usepackage[english]{babel}
\usepackage{lipsum}
\usepackage{mathtools}
\usepackage{verbatim}
\usepackage{multirow}
\usepackage{subfigure}
\usepackage[colorlinks=true]{hyperref}
\usepackage{mathrsfs}
\hypersetup{linkcolor=blue}
\hypersetup{citecolor=green}

\usepackage{cleveref}
\crefformat{section}{\S#2#1#3} 
\crefformat{subsection}{\S#2#1#3}
\crefformat{subsubsection}{\S#2#1#3}

\hypersetup{final}

\newcommand{\vect}[1]{\boldsymbol{#1}}

\newtheorem{Thm}{Theorem}
\newtheorem{Cor}[Thm]{Corollary}
\newtheorem{Lemma}[Thm]{Lemma}
\newtheorem{Def}[Thm]{Definition}
\newtheorem{Prop}[Thm]{Proposition}

\newtheorem{Rmk}[Thm]{Remark}

\numberwithin{equation}{section}

\newcommand{\I}{\mathcal{I}}

\newcommand{\xb}{\bold{x}}
\newcommand{\xbvt}{\tilde{\vect{x}}}
\newcommand{\db}{\vect{d}}
\newcommand{\Xb}{\vect{X}}

\newcommand{\yb}{\bold{y}}

\newcommand{\Gt}{\tilde{G}}
\newcommand{\Ht}{\tilde{H}}

\newcommand{\R}{\mathbb{R}}

\newcommand{\E}{\mathbb{E}}

\newcommand{\N}{\mathbb{N}}
\newcommand{\Sbb}{\mathbb{S}}

\newcommand{\ubar}{\vect{\bar{u}}}
\newcommand{\ut}{\vect{\tilde{u}}}
\newcommand{\Sb}{\bold{S}}
\newcommand{\Sbt}{\tilde{\bold{S}}}
\newcommand{\Sbh}{\hat{\bold{S}}}
\newcommand{\Ab}{\vect{A}}

\newcommand{\Abwt}{\widetilde{\Ab}}
\newcommand{\Abh}{\hat{\Ab}}
\newcommand{\Bb}{\vect{B}}

\newcommand{\Bbwt}{\widetilde{\Bb}}
\newcommand{\Bbh}{\hat{\Bb}}

\newcommand{\Cbv}{\vect{C}}
\newcommand{\Dbv}{\vect{D}}
\newcommand{\Fbv}{\vect{F}}
\newcommand{\Vbv}{\vect{V}}
\newcommand{\Rbv}{\vect{R}}
\newcommand{\Ybv}{\vect{Y}}
\newcommand{\Db}{\bold{D}}
\newcommand{\Ib}{\bold{I}}

\newcommand{\Lb}{\vect{L}}
\newcommand{\Pb}{\vect{P}}
\newcommand{\Lbt}{\tilde{\vect{L}}}
\newcommand{\Eb}{\bold{E}}
\newcommand{\Mb}{\bold{M}}
\newcommand{\Qb}{\bold{Q}}
\newcommand{\Rb}{\bold{R}}

\newcommand{\epstil}{\tilde{\varepsilon}}
\newcommand{\pih}{\hat{\pi}}

\newcommand{\piwt}{{\widetilde{\pi}}}
\newcommand{\Pib}{{\bold{\Pi}}}
\newcommand{\Pibh}{{\hat{\Pib}}}
\newcommand{\Omb}{\bold{\Omega}}

\newcommand{\bb}{\bold{b}}
\newcommand{\eb}{\bold{e}}
\newcommand{\rpm}{\raisebox{.2ex}{$\scriptstyle\pm$}}

\newcommand{\js}{\mathscr{J}}

\newcommand{\ellb}{\bar{\ell}}
\newcommand{\ellt}{\tilde{\ell}}
\newcommand{\Phib}{\bold{\bar{\Phi}}}
\newcommand{\Phit}{\bold{\tilde{\Phi}}}
\newcommand{\Ubv}{\vect{U}}
\newcommand{\UA}{\Ubv_{\Ab}}
\newcommand{\UB}{\Ubv_{\Bb}}
\newcommand{\betat}{\tilde{\beta}}
\newcommand{\Hbh}{\hat{\bold{H}}}
\newcommand{\UBo}{\overline{\UB}}
\DeclareMathOperator{\tr}{tr}
\DeclareMathOperator{\rank}{rank}
\DeclareMathOperator{\diag}{diag}
\newcommand{\errAB}{\mathrm{err}_{\Ab,\Bb}}

\newcommand{\Bvb}{\vect{B}}

\definecolor{LightBlue}{rgb}{0, 0.7, .7}
\definecolor{Green1}{rgb}{0.0, 0.5, 0.0}
\definecolor{green}{rgb}{0.0, 0.42, 0.24}
\definecolor{byzantine}{rgb}{0.74, 0.2, 0.64}

\algrenewcommand{\Return}{\State\algorithmicreturn~}

\newcommand{\ind}{\text{\color{white}.$\quad$}}

\makeatletter
\def\legendre@dash#1#2{\hb@xt@#1{%
  \kern-#2\p@
  \cleaders\hbox{\kern.5\p@
    \vrule\@height.2\p@\@depth.2\p@\@width\p@
    \kern.5\p@}\hfil
  \kern-#2\p@
  }}
\def\@legendre#1#2#3#4#5{\mathopen{}\left(
  \sbox\z@{$\genfrac{}{}{0pt}{#1}{#3#4}{#3#5}$}%
  \dimen@=\wd\z@
  \kern-\p@\vcenter{\box0}\kern-\dimen@\vcenter{\legendre@dash\dimen@{#2}}\kern-\p@
  \right)\mathclose{}}
\newcommand\legendre[2]{\mathchoice
  {\@legendre{0}{1}{}{#1}{#2}}
  {\@legendre{1}{.5}{\vphantom{1}}{#1}{#2}}
  {\@legendre{2}{0}{\vphantom{1}}{#1}{#2}}
  {\@legendre{3}{0}{\vphantom{1}}{#1}{#2}}
}
\def\dlegendre{\@legendre{0}{1}{}}
\def\tlegendre{\@legendre{1}{0.5}{\vphantom{1}}}
\makeatother

\DeclareMathOperator*{\argmin}{arg\,min}

\newcommand{\SVD}{\mathrm{SVD}}
\newcommand{\QR}{\mathrm{QR}}

\newcommand{\crmm}{CR\mathrm{-MM}}

\newcommand{\sfsty}[1]{\ensuremath{\mathsf{#1}}}  
\newcommand{\bmm}{\sfsty{BasicMatrixMultiplication}}

\usepackage[margin=1in,footskip=0.25in]{geometry}

\title{Approximate Matrix Multiplication through Joint Leverage Score Sampling}
\author{\large%
Neophytos Charalambides, Mert Pilanci, and Alfred O. Hero III
  \vspace{-.25in}
  \thanks{This work was initially presented in \cite{Neo23}.}
}

\date{\vspace{-5ex}}

\begin{document}

\maketitle

\begin{abstract}

In this paper, we investigate approximate matrix multiplication through random sampling. Specifically, we propose using a non-uniform sampling distribution constructed from the leverage scores of the two matrices, and characterize the associated approximation error in terms of the matrix spectral norm; which is well-known to better characterize the downstream error when the matrix product operator is applied to an input vector. We provide the probability error bounds for the proposed algorithm are also given, and discuss connections to data-oblivious approximate matrix multiplication, and graph sparsification algorithms.
\end{abstract}

\maketitle

\section{Introduction}
\label{intro}

$\ind$ Matrix multiplication is one of the key underlying operations used in  algorithms developed in numerical analysis, scientific computing, engineering, statistics, computer science, data science and machine learning. In many applications, \textit{e.g.} training contemporary machine learning and statistical inference models, this operation is executed hundreds or thousands of times, making it a bottleneck and an impediment in large scale computations.\\
$\ind$ Considering square matrices of order $N$, naive matrix multiplication (MM) requires $O(N^3)$ operations. More elegant algorithms exist, which require $O(N^{\omega})$ operations, for $\omega<2.81$ \cite{Str69,CW87,Gal14,AV21,DWZ23}. The smallest known MM exponent $\omega<2.3716$, was recently achieved in \cite{WXXZ23}. A drawback of these faster algorithms though, is that the hidden constants in their asymptotic analysis are excessively large, making them impractical.\\
$\ind$ In this paper, we utilize ideas from the ``$\ell_2$-\textit{subspace embedding}'' problem, to devise a novel \textit{approximate matrix multiplication} algorithm. We propose a randomized approach that judiciously subsamples pairs of rows from $\Ab$ and $\Bb$, to approximate the matrix product
$$ \overbrace{\begin{pmatrix} & & & & & & \\ & & & \Ab^\top & & & \\ & & & & & & \end{pmatrix}}^{L\times N} \cdot \overbrace{\begin{pmatrix} & & \\ & & \\ & & \\ & \Bb & \\ & & \\ & & \\ & & \end{pmatrix}}^{N\times M} \approx \overbrace{\begin{pmatrix} & & & & \\ & & \Abh^\top & & \\ & & & & \end{pmatrix}}^{L\times r} \cdot \overbrace{\begin{pmatrix} & & \\ & & \\ & \Bbh & \\ & & \\ & & \end{pmatrix}}^{r\times M} $$
where $\Abh=\Sb\Ab$ and $\Bbh=\Sb\Bb$, for an explicit \textit{sketching matrix} $\Sb\in\R^{r\times N}$ with $r\ll N$. For our proposed $\Sb$, we define a sampling distribution based on the left leverage scores of both $\Ab$ and $\Bb$, which we call `\textit{joint leverage scores distribution}', and is a generalization of the matrix leverage scores distribution. These scores reveal influential pairs of rows across the two matrices. It is surprising that even though leverage scores are the main technique and tool used to study and devise subspace embeddings; and have been utilized in many applications, they have not been been used for the most common operation in linear algebra; matrix multiplication. A non-exhaustive list of where they have been utilized are the block Kaczmarz method \cite{NT14,LR24}, low-rank approximations \cite{DM08,CYD19}, and $k$-means clustering \cite{BZMD14,CEMMP15,Cha20}.\\
$\ind$ The proposed algorithm is similar to the $\bmm$ approximate matrix multiplication (AMM) algorithm \cite{DK01,DKM06a}, and is motivated by the leverage score sampling algorithm for $\ell_2$-subspace embedding \cite{DMM06,DMMS11,MMY15,CPH20a}. For reasons that will become clear, we will refer to $\bmm$ as $\crmm$. Both $\crmm$ and leverage score sampling have been extensively studied in the literature of randomized numerical linear algebra (RandNLA) \cite{Mah16,Woo14}. The only difference between our algorithm and $\crmm$, is the sampling distribution.\\
$\ind$ A benefit of sampling according to leverage scores, is that it permits the derivation of a spectral bound on the approximation error of the AMM, analogous to that of the $\ell_2$-subspace embedding method \cite{DMMS11,Woo14,Mah16}. Such spectral guarantees have not been derived for the general case of $\crmm$, whose error has been quantified in terms of the Frobenius norm. One can use similar techniques to ours to derive a spectral bound for the $\crmm$ algorithm for the case where $\Ab\neq\Bb$, though the resulting sampling complexity could be very high. We discuss this in Subsection \ref{CR_subsec}. The special case where $\Ab=\Bb$ \cite{DMMS11,Mah16} has been studied. This motivates our work, in which we propose an importance sampling algorithm similar to $\crmm$, which is also used to derive a spectral bound guarantee for $\crmm$. It is worth noting that leverage scores have also been used in graph sparsification algorithms \cite{SS11,DM10}, and that block variants of sampling sketching algorithms have also been extensively studied \cite{OJXE18,Wu19,CMPH23,CPH23b,YM23,NL23}.\\
$\ind$ Our work on AMM is closely related to that of \cite{CNW16}. Specifically, they give a similar spectral AMM characterization to ours, both of which imply the same spectral error bound (Lemma \ref{lem_error}). The work of \cite{CNW16} focuses on the so-called ``\textit{oblivious subspace embeddings}'' (OSE) and their implications; that satisfy certain moment bounds, which they argue (using a ``black box theorem'') yields an AMM with optimal properties. However, they do not give an explicit AMM algorithm. Similar results have also been derived independently in \cite{Sar06,KVZ14,MMG16}, which primarily deal with OSEs and give guarantees in terms of the matrices' stable rank. In contrast, instead of relying on OSE black box arguments and the ``\textit{JL-moment property}'', we take a more direct approach by using a generalization of the ``$\ell_2$-\textit{subspace embedding property}''  to characterize the accuracy of our proposed AMM algorithm.\\
$\ind$ As in \cite{KVZ14,CNW16}, any AMM sketching matrix which satisfies our condition, is a $\ell_2$-subspace embedding sketching matrix. This is achieved by simply applying it to the case where $\Bb=\Ab$, in which case we transition from $\Ab^\top\Ab$ to the least squares objective function $\|\Ab\xb-\bb\|_2^2$. For the algorithm we propose; which is not data-oblivious\footnote{By ``\textit{oblivious}'' we mean that information is do not directly extracted from the data, to determine the sketching matrix. In the case of leverage score sampling, we use the left orthonormal basis (in our case bases), to determine the sampling distribution. Therefore, approaches which explicitly utilize leverage scores are not data-oblivious methods.}, the resulting embedding is the leverage score $\ell_2$-subspace embedding sketch \cite{DMMS11}.\\
$\ind$ Before transitioning to the main section of this paper, we briefly discuss the importance of sketching, which has been well-studied over the past two decades. Sketching and RandNLA is an interdisciplinary field that exploits randomization as a computational resource; to develop improved algorithms for large-scale linear algebra problems, while as a technique lies in the intersection of theoretical computer science and information theory. Primarily studied through the computer science community; sketching is used to compress information, which is the primary objective of \textit{information theory} and \textit{source coding}. This gives two vastly different perspectives to the problem of data compression. Furthermore, sketching has been studied in various applications, \textit{e.g.} matrix and tensor algorithms \cite{LWMRT07,MD09,CLMMPS15,GGMW20,HH21,CJ21,MS21,CPH23b}, graph algorithms \cite{KK15,ACKQWZ16,AALG18,AKK19,HLMM20}, and of course AMM \cite{DK01,DKM06a,DKM06b,Pag12,Wu19,YM23}. Furthermore, exact MM has also been studied extensively through different perspectives, \textit{e.g.} \cite{Str69,CW90,Raz02,Sto10,Gal14,AW18,MNSUW18,AV21,BCGPU23,WXXZ23,DWZ23}.\\
$\ind$ The paper is organized as follows. In Section \ref{CR_mult} we first recall the notion of leverage scores, and the $\crmm$ algorithm. Next, we describe our sampling approach to AMM; which we call `\textit{joint leverage score sampling}'. Then, we present our AMM algorithm. Our spectral characterization of AMM is derived in Subsection \ref{subsec_AMM_char}. In Subsection \ref{guar_sec} we present the theory behind our approach, and our spectral guarantees. In Section \ref{jls_other} we discuss other extensions of our approach, which suggest worthwhile future work. We conclude this paper with Section \ref{concl_sec}.

\section{Joint Leverage Score Multiplication}
\label{CR_mult}

\subsection{Preliminaries}

$\ind$ We first recall the $\crmm$ approximation algorithm. Consider the two matrices $\Ab\in\R^{N\times L}$ and $\Bb\in\R^{N\times M}$, for which we want to approximate the product $\Ab^\top\Bb$. It is known that the product may be approximated by subsampling with replacement (w.r.) row pairs from $\Ab$ and $\Bb$, where the sampling probabilities are proportional to their $\ell_2$-norms \cite{DK01}. That is, we sample with replacement $r$ pairs $(\Ab_{(i)},\Bb_{(i)})$\footnote{We denote the $i^{th}$ row of matrix $\Mb$ by $\Mb_{(i)}$.} for $i\in\N_N\coloneqq\{1,\cdots,N\}$ and $r\ll N$, with probability
\begin{equation*}
\label{CR_distr}
  p_i = \|\Ab_{(i)}\|_2\cdot\|\Bb_{(i)}\|_2 \big/ \phi
\end{equation*}
for $\phi\coloneqq\sum_{j=1}^N\|\Ab_{(j)}\|_2\cdot\|\Bb_{(j)}\|_2$, and sum a rescaling of the samples' outer-products:
\begin{equation}
\label{CR_eq}
  \Ab^\top\Bb \approx \frac{1}{r}\cdot \left(\sum_{j\in\I}\frac{1}{p_j}\Ab_{(j)}^\top\Bb_{(j)}\right) = \sum_{j\in\I}\frac{\Ab_{(j)}^\top}{\sqrt{rp_j}}\cdot\frac{\Bb_{(j)}}{\sqrt{rp_j}} \eqqcolon \Ybv
\end{equation}
where $\I$ is the multiset consisting of the indices (possibly repeated) of the sampled pairs, \textit{i.e.} $|\I|=r$.\\
$\ind$ Denote the corresponding sketches of $\Ab^\top$ and $\Bb$ in $\crmm$ by $\Cbv\in\R^{L\times r}$ and $\Rbv\in\R^{r\times M}$ respectively. The approximation \eqref{CR_eq} satisfies $\|\Ab^\top\Bb-\Cbv\Rbv\|_F=O(\|\Ab\|_F\|\Bb\|_F/\sqrt{r})$. Further details on this algorithm can be found in \cite{DK01,DKM06a,DKM06b,Woo14,Mah16,CPH20b,NL21}. We summarize what is known about the $\crmm$ algorithm in Theorem \ref{thm_CR}.

\begin{Thm}[\cite{Mah16} Section 3.2]
\label{thm_CR}
\textit{The estimator $\Ybv=\Cbv\Rbv$ from \eqref{CR_eq} is unbiased, while the sampling probabilities $\{p_i\}_{i=1}^N$ minimize the variance, \textit{i.e.}
\begin{equation}
\label{eq_min_var_CR}
  \{p_i\}_{i=1}^N = \argmin_{\substack{\sum_{i=1}^Np_i=1\\ p_i\geqslant0 \ \forall i}} \Big\{\mathrm{Var}(\Ybv)=\E\left[\|\Ab^\top\Bb-\Cbv\Rbv\|_F^2\right]\Big\}
\end{equation}
and $\Ybv$ is an $\epsilon$-multiplicative error approximation of the matrix product, with high probability. Specifically, for $\delta\geqslant0$ and $r\geqslant \frac{1}{\delta^2\epsilon^2}$ the number of sampling trials which take place:
\begin{equation}
\label{CR_mult_error}
  \Pr\big[\|\Ab^\top\Bb-\Cbv\Rbv\|_F \leqslant\epsilon\cdot\|\Ab\|_F\|\Bb\|_F\big]\geqslant1-\delta
\end{equation}
for any $\epsilon>0$.}
\end{Thm}

$\ind$ The special case where $\Ab=\Bb$ in the $\crmm$ algorithm has also been studied, as it appears in numerous applications. This restriction allows us to use statements from random matrix theory \cite{Oli10}, to get stronger spectral norm bounds. One such bound is presented in Theorem \ref{sp_bd_thm}.

\begin{Thm}[\cite{DMMS11,Mah16}]
\label{sp_bd_thm}
\textit{Let $\Ab\in\R^{N\times L}$ with $\sigma_{\max}(\Ab)=\|\Ab\|_2\leqslant1$, and approximate the product $\Ybv\approx \Ab^\top\Ab$ through $\crmm$. Let $\epsilon\in(0,1)$ be an accuracy parameter, and assume that $\|\Ab\|_F^2\geqslant1/24$. If
\begin{equation}
\label{samples_AAt}
  r\geqslant\frac{96\|\Ab\|_F^2}{\epsilon^2}\ln\left(\frac{96\|\Ab\|_F^2}{\epsilon^2\sqrt{\delta}}\right)\geqslant\frac{4}{\epsilon^2}\ln\left(\frac{4}{\epsilon^2\sqrt{\delta}}\right)
\end{equation}
for $r\leqslant N$, then
\begin{equation}
\label{bd_AAt}
  \Pr\left[\|\Ab^\top\Ab-\Ybv\|_2\leqslant\epsilon\right] \geqslant 1-\delta\ .
\end{equation}
}
\end{Thm}

$\ind$ The left \textit{leverage scores} $\{\ell_i\}_{i=1}^N$ of $\Mb\in\R^{N\times d}$ are defined as the diagonal entries of the projection $\Pb_\Mb=\Mb\Mb^{\dagger}$, \textit{i.e.} $\ell_i=(\Pb_{\Mb})_{ii}=\|\Mb_{(i)}\|_{\Mb^{\dagger}}^2$. Equivalently, considering the reduced left singular vectors matrix $\Ubv\in\R^{N\times d}$ of $\Mb$, we have $\ell_i=\|\Ubv_{(i)}\|_2^2$. A sampling distribution is defined over the rows of $\Mb$ by normalizing the scores $\{\ell_i\}_{i=1}^N$, where each row of $\Mb$ has respective probability of being sampled $p_i=\ell_i\big/\big(\sum_{j=1}^N\ell_j\big)=\ell_i/\|\Ubv\|_F^2=\ell_i/d$. It is a known fact that
\begin{equation}
\label{sum_lvg_scores}
  \sum_{i=1}^N\ell_i=\tr(\Pb_\Mb)=\tr(\Ubv\Ubv^\top)=\rank(\Mb)\leqslant d\ .
\end{equation}

$\ind$ We note that methods for approximating the leverage scores are available \cite{DMMW12,RCCR18,SS11}, which do not require directly computing the left orthonormal basis $\Ubv$. Furthermore, any orthonormal basis of the column-space of $\Mb$ would suffice to deduce the leverage scores. We will be working with the left singular vectors matrix, to simplify certain expressions.\\
$\ind$ Lastly, we recall the following noncommutative Bernstein matrix inequality, which we use to prove the spectral characterization of our AMM algorithm. We slightly paraphrase the theorem, to better suit our notation.

\begin{Thm}[\cite{Mag10,Rec11}]
\label{Bern_ineq}
\textit{Let $\Xb_1,\cdots,\Xb_r$ be independent copies of a zero-mean random matrix $\Xb\in\R^{d_1\times d_2}$, with $\|\Xb\|_2\leqslant\gamma$ and $\max\big\{\big\|\E\big[\Xb^\top\Xb\big]\big\|_2,\big\|\E\big[\Xb\Xb^\top\big]\big\|_2\big\}\leqslant\sigma^2$. Then, $\forall\epsilon>0$:
\begin{equation*}
\label{matr_Chern_expr}
  \Pr\left[\left\|\frac{1}{r}\sum_{l=1}^r\Xb_l\right\|_2>\epsilon\right]\leqslant(d_1+d_2)\cdot\exp\left(\frac{-r\epsilon^2/2}{\sigma^2+\gamma\epsilon/3}\right).
\end{equation*}
}
\end{Thm}

\subsection{Joint Leverage Score Sampling}
\label{subsec_jls}

$\ind$ We define the \textit{joint leverage score} of the pair $(\Ab_{(i)},\Bb_{(j)})$ as the geometric mean of corresponding leverage scores, \textit{i.e.}
$$ \js_{i,j} = \sqrt{\ellb_i\cdot\ellt_j} $$
where $\ellb_i$ and $\ellt_i$ are respectively the $i^{th}$ and $j^{th}$ leverage score of $\Ab$ and $\Bb$.\footnote{To abbreviate notation throughout the paper, we will use the ``$-$'' and ``$\sim$'' (over-scripts) to denote corresponding quantities of $\Ab$ and $\Bb$ respectively.} By $\UA\in\R^{N\times L}$ and $\UB\in\R^{N\times M}$ we denote the reduced left singular vectors matrices of $\Ab$ and $\Bb$, with respective rows $\big\{\ubar_i\coloneqq(\UA)_{(i)}\big\}_{i=1}^N\subsetneq\R^{1\times L}$ and $\big\{\ut_i\coloneqq(\UB)_{(i)}\big\}_{i=1}^N\subsetneq\R^{1\times M}$. By the definition of leverage scores, it follows that
$$ \ellb_i = \ubar_i\ubar_i^\top = (\UA)_{(i)}(\UA)_{(i)}^\top \quad \text{ and } \quad \ellt_i = \ut_i\ut_i^\top = (\UB)_{(i)}(\UB)_{(i)}^\top, \quad \text{for all } i\in\N_N.$$
$\ind$ Let $\diag(\Mb)$ denote the vector of $\Mb$'s diagonal entries, and $\sqrt{\db}$ the vector comprised of the positive square roots of $\db$'s entries. For $\Pb_{\Ab}=\UA\UA^\top$ and $\Pb_{\Bb}=\UB\UB^\top$ the respective projection matrices of $\Ab$ and $\Bb$, the vectors $\db_{\Ab}\coloneqq\sqrt{\diag(\Pb_{\Ab})}$ and $\db_{\Bb}\coloneqq\sqrt{\diag(\Pb_{\Bb})}$ are comprised of the square roots of the leverage scores of $\Ab$ and $\Bb$, hence $\|\db_{\Ab}\|_2=\big(\sum_{l=1}^N\ellb_l\big)^{1/2}$ and $\|\db_{\Bb}\|_2=\big(\sum_{l=1}^N\ellt_l\big)^{1/2}$. Let $\mu\coloneqq\sqrt{LM}$. By the Cauchy–Schwarz inequality and \eqref{sum_lvg_scores}:
\begin{equation}
\label{sum_jlvg_scores}
  \tau\coloneqq\sum_{l=1}^N\js_{l,l} = \langle\db_{\Ab},\db_{\Bb}\rangle \leqslant \|\db_{\Ab}\|_2\cdot\|\db_{\Bb}\|_2 \leqslant \sqrt{LM}\quad \implies \quad \tau\leqslant\mu\ .
\end{equation}
$\ind$ In this work we sample pairs of rows of the same index, and do not consider pairs $(\Ab_{(i)},\Bb_{(j)})$ for $i\neq j$. Hence, we only consider the joint leverage scores $\left\{\js_{i,i}\right\}_{i=1}^N$, which we normalize to determine the sampling probabilities
\begin{equation}
\label{sampling_probs}
  \pi_i \coloneqq \js_{i,i}\big/\tau\ , \quad \text{ for each } i\in\N_N.
\end{equation}
$\ind$ Similar to the $\crmm$ and leverage score sampling algorithms, we construct a sketching matrix $\Pib\in\R^{r\times N}$ to reduce the effective dimension from $N$ to $r$, and approximate the product:
$$ \Ab^\top\Bb \approx (\Pib\Ab)^\top(\Pib\Bb) \eqqcolon \Abh^\top\Bbh \ . $$
Once the sketches $\Abh,\Bbh$ are known, computing $\Abh^\top\Bbh$ requires $O(LMr)$ operations through naive MM. The pseudocode of our approach is presented in Algorithm \ref{alg_jlvg}.\\
$\ind$ Note that any algorithm that approximates leverage scores \cite{DMMW12,RCCR18,SS11} could be used to speed up the computation of $\{\pi_i\}_{i=1}^N$, as determining them exactly is a cumbersome task. The guarantees we give will assume that the sampling distribution used in Algorithm \ref{alg_jlvg} is approximate.

\begin{algorithm}[h]
\label{alg_jlvg}
\SetAlgoLined
  \KwIn{$\Ab\in\R^{N\times L}$, $\Bb\in\R^{N\times M}$, $r<N$}
  \KwOut{Approximate product $\Abh^\top\Bbh\approx\Ab^\top\Bb$}
  \textbf{Initialize:} $\Pibh=\bold{0}_{r\times N}$\\
  \textbf{Determine:} approximate distribution $\{\pih_l\}_{l=1}^N$ to \eqref{sampling_probs}\\
  \For{$j=1$ to $r$}
  {
    sample w.r. $i(j)$ from $\N_N$, according to $\{\pih_l\}_{l=1}^N$\\
      $\Pibh_{j,i(j)}=1/\sqrt{r\pih_{i(j)}}$
  }
 \Return $\Abh^\top\Bbh = (\Pib\Ab)^\top(\Pib\Bb)$
 \caption{Joint Leverage Score Sampling AMM}
\end{algorithm}

$\ind$ Using probabilities exactly equal to $\{\pi_i\}_{i=1}^N$ is an overkill. Similar guarantees hold if we instead use approximate probabilities $\{\pih_i\}_{i=1}^N$ that are ``close'' in the following sense: $\pih_i\geqslant\beta\pi_i$ for all $i$, where $\beta\in(0,1]$ is a misestimation factor. This factor quantifies how far  $\{\pih_i\}_{i=1}^N$ is from $\{\pi_i\}_{i=1}^N$. If we sample according to the approximate distribution, the same embedding result will follow, if we modify the sampling complexity by a factor of $1/\beta$, \textit{i.e.} we oversample. The key point here is that it is essential not to underestimate high joint leverage score row pairs too much. It is however acceptable if we overestimate and thus oversample some low score pairs, as long as it is not done excessively. Throughout this paper, we will denote the resulting sketching matrices of Algorithm \ref{alg_jlvg} by $\Pib$ and $\Pibh$, when sampling takes place according to $\{\pi_i\}_{i=1}^N$ and $\{\pih_i\}_{i=1}^N$ respectively.

\subsection{Spectral Characterization for AMM}
\label{subsec_AMM_char}

$\ind$ Next, we recall the basic notion of a $\ell_2$-\textit{subspace embedding} for the
column-space of a $N\times d$ matrix $\Mb$, which is a powerful hammer for solving least squares regression.

\begin{Def}[\cite{Woo14}]
\label{def_subsp_emb}
\textit{A sketching matrix $\Sb\in\R^{r\times N}$ is a \textbf{\textit{$\ell_2$-subspace embedding}} of the column-space of $\Mb\in\R^{N\times d}$ for $N\gg d$ and $N>r>d$; with a reduced orthonormal basis $\Ubv\in\R^{N\times d}$ such that $\Ubv^\top\Ubv=\Ib_d$, if it satisfies
\begin{equation}
\label{se_char}  
  \|\Ib_d-(\Sb\Ubv)^\top(\Sb\Ubv)\|_2\leqslant \epsilon
\end{equation}
for $\epsilon>0$, with high probability.}
\end{Def}

$\ind$ In this subsection, we derive an analogous characterization for AMM, in order to quantify the error of our approximation. Considering the reduced $\SVD$s of $\Ab$ and $\Bb$, an AMM sketching matrix $\Sb\in\R^{r\times N}$ ought to satisfy
\begin{align*}
  \Ab^\top\cdot\Bb = \Vbv_{\Ab}\vect{\Sigma}_{\Ab}\UA^\top\cdot\UB\vect{\Sigma}_{\Bb} \Vbv_{\Bb}^\top
  \approx \Vbv_{\Ab}\vect{\Sigma}_{\Ab}\UA^\top\Sb^\top\cdot\Sb\UB\vect{\Sigma}_{\Bb}\Vbv_{\Bb}^\top
  = (\Sb\Ab)^\top\cdot(\Sb\Bb)\ .
\end{align*}
Hence, a sufficient condition that implies $\Ab^\top\Bb\approx(\Sb\Ab)^\top(\Sb\Bb)$ is $\UA^\top\UB \approx (\Sb\UA)^\top(\Sb\UB)$, which can be formulated as
\begin{equation}
\label{sp_AMM_char}  
  \|\UA^\top\UB - (\Sb\UA)^\top(\Sb\UB)\|_2\leqslant\epsilon
\end{equation}
for $\epsilon>0$ a small constant. We denote $\Dbv_\Sb=\big(\UA^\top\UB - (\Sb\UA)^\top(\Sb\UB)\big)$, and
\begin{equation}
\label{Delta_AB}
  \Delta_{\Ab,\Bb}^\Sb = \Ab^\top\Bb-\Abh^\top\Bbh =  \Vbv_{\Ab}\vect{\Sigma}_{\Ab}\cdot \Dbv_\Sb\cdot\vect{\Sigma}_{\Bb}\Vbv_{\Bb}^\top\ .
\end{equation}

\begin{Lemma}
\label{lem_error}
\textit{If $\Sb$ satisfies \eqref{sp_AMM_char}, then $\errAB^{\ell_2,\Sb}\coloneqq\big\|\Delta^\Sb_{\Ab,\Bb}\big\|_2\leqslant\epsilon\cdot\sigma_1(\Ab)\sigma_1(\Bb)$.}
\end{Lemma}

\begin{proof}
This is a direct application of the submultiplicativity property of the $\ell_2$-norm to \eqref{Delta_AB}.
\end{proof}

\begin{Rmk}
\textit{If $\Bb=\Ab$, the sampling probabilities $\{\pi_i\}_{i=1}^N$ are equal to the normalized leverage scores of $\Ab$, and condition \eqref{sp_AMM_char} results in the $\ell_2$-subspace embedding property \eqref{se_char}. Furthermore, $\Pib$ would be the same sketching matrix as that of the leverage score $\ell_2$-subspace embedding sampling algorithm. Additionally, for $\Ab$, $\Bb$ both orthonormal matrices, the output of Algorithm \ref{alg_jlvg} is the same as that of the $\crmm$ algorithm.}
\end{Rmk}

\begin{Prop}
\label{prop_amm_se}
\textit{Any algorithm that constructs a sketching matrix $\Sb$ which satisfies \eqref{sp_AMM_char}, can be used to construct a $\ell_2$-subspace embedding sketch of $\Ab$.}
\end{Prop}

\begin{proof}
Assume we have an algorithm that when given $\Ab$ and $\Bb$, produces $\Sb$ which satisfies \eqref{sp_AMM_char}. Further assume it is only given $\Ab$ as an input, and let $\Bb=\Ab$. This algorithm then constructs $\Sb$ for the case where $\UB=\UA$. Thus $\UA^\top\UB=\Ib_d$, and the resulting $\Sb$ satisfies \eqref{se_char}.
\end{proof}

\subsection{Approximation Guarantees}
\label{guar_sec}

$\ind$ In this subsection, we present our theoretical results of Algorithm \ref{alg_jlvg}. Unlike the guarantees of the $\crmm$ algorithm for $\Ab\neq\Bb$, whose error is quantified in terms of the Frobenius norm; the guarantee we give is in terms of the spectral norm. Specifically, we show that \eqref{sp_AMM_char} is achieved with a high probability.\\
$\ind$ To prove our main theorem, we first need to determine the expectation; and bound the spectral norm and variance of appropriately defined matrix random variables. Once these quantities are determined, we can directly apply Theorem \ref{Bern_ineq}.\\
$\ind$ For $i(j)\in\N_N$, which denotes the index of the sampled pair $\big(\Ab_{(i(j))},\Bb_{(i(j))}\big)$ at the $j^{th}$ sampling trial of Algorithm \ref{alg_jlvg}, we define
\begin{equation}
\label{def_mrv}  
  \Xb_{i(j)}\coloneqq\left(\UA^\top\UB-\frac{\ubar_{i(j)}^\top\ut_{i(j)}}{\pih_{i(j)}}\right)\in\R^{L\times M}\ .
\end{equation}
$\ind$ The matrices $\big\{\Xb_{i(j)}\big\}_{j=1}^r$ are independent copies of the matrix random variable $\Xb$ we are considering, since we are sampling with replacement. Note that the realizations $\Xb_\iota$ of $\Xb$ correspond to $\Dbv_\Pibh$; \textit{i.e.} $\Dbv_\Sb$ for $\Sb\gets\Pibh$, with $r=1$.

\begin{Lemma}
\label{lem_1}
\textit{For $\Xb_{i(j)}$ as defined in \eqref{def_mrv}, we have $\E\left[\Xb_{i(j)}\right]=\bold{0}_{M\times L}$ and $\|\Xb_{i(j)}\|_2\leqslant1+\mu/\beta$.}
\end{Lemma}

\begin{proof}
A straightforward computation gives us
\begin{align*}
  \E\left[\Xb_{i(j)}\right] = \UA^\top\UB - \sum_{l=1}^N\pih_l\cdot\left(\frac{\ubar_l^\top\ut_l}{\pih_l}\right)
  = \UA^\top\UB - \sum_{l=1}^N\ubar_l^\top\ut_l = \UA^\top\UB-\UA^\top\UB
\end{align*}
thus $\E\left[\Xb_{i(j)}\right]=\bold{0}_{L\times M}$.
For the second part of the lemma, we have
\begin{equation*}
  \|\Xb_{i(j)}\|_2 \leqslant \|\UA^\top\UB\|_2 + \frac{1}{\pih_{i(j)}}\cdot\|\ubar_{i(j)}^\top\ut_{i(j)}\|_2
\end{equation*}
for which we observe that
\begin{align*}
  \|\ubar_{i(j)}^\top\ut_{i(j)}\|_2 &\leqslant \|\ubar_{i(j)}^\top\ut_{i(j)}\|_F\\
  &= \left(\tr\big(\ut_{i(j)}^\top\ubar_{i(j)}\cdot\ubar_{i(j)}^\top\ut_{i(j)}\big)\right)^{1/2}\\
  &= \left(\tr\big(\ubar_{i(j)}\ubar_{i(j)}^\top\cdot\ut_{i(j)}\ut_{i(j)}^\top\big)\right)^{1/2}\\
  &= \sqrt{\ellb_{i(j)}\cdot\ellt_{i(j)}}\\
  &= \js_{i(j),i(j)}
\end{align*}
and
$$ \|\UA^\top\UB\|_2 \leqslant \|\UA\|_2\cdot\|\UB\|_2 = 1 \ .$$
By combining these two observations, we conclude that
\begin{equation*}
  \|\Xb_{i(j)}\|_2 \leqslant 1+\frac{\js_{i(j),i(j)}}{\pih_{i(j)}} = 1+\frac{\tau\cdot\pi_{i(j)}}{\pih_{i(j)}} \leqslant 1+\frac{\tau}{\beta} \leqslant 1+\frac{\mu}{\beta}\ ,
\end{equation*}
where in the last inequality we invoked \eqref{sum_jlvg_scores}.
\end{proof}

$\ind$ In order to upper bound the variance parameter of the random matrix variable defined in \eqref{def_mrv}, we define the matrices 
\begin{equation*}
  \Phit=\sum_{l=1}^N\sqrt{\ellb_l/\ellt_l}\cdot\ut_l^\top\ut_l \qquad \text{ and } \qquad \Phib=\sum_{l=1}^N\sqrt{\ellt_l/\ellb_l}\cdot\ubar_l^\top\ubar_l
\end{equation*}
where $\Phit\in\R^{M\times M}$ and $\Phib\in\R^{L\times L}$. Furthermore, to simplify our notation, let $\alpha_i\coloneqq\sqrt{\ellb_i/\ellt_i}$ for each $i\in\N_N$. 
We note that $\{\ut_i\}_{i=1}^N$ and $\{\ubar_i\}_{i=1}^N$ are respectively linearly dependent spanning sets of both the column and row spaces of $\Phit$ and $\Phib$ respectively, when $N>M,L$. This prohibits us from explicitly computing the spectral norms of $\Phit$ and $\Phib$, though an upper bound suffices for our purposes.

\begin{Lemma}
\label{lem_2}
\textit{For $\Xb_{i(j)}$ as defined in \eqref{def_mrv}; and $\iota\gets i(j)$, we have $\big\|\E\big[\Xb_\iota^\top\Xb_\iota\big]\big\|_2\leqslant\frac{\tau}{\beta}\cdot\lambda_1(\Phit)+1$ and $\big\|\E\big[\Xb_\iota\Xb_\iota^\top\big]\big\|_2\leqslant\frac{\tau}{\beta}\cdot\lambda_1(\Phib)+1$.}
\end{Lemma}

\begin{proof}
By a straightforward computation
\begin{align*}
  \Xb_\iota^\top\Xb_\iota &= \left(\UA^\top\UB-\frac{\ubar_\iota^\top\ut_\iota}{\pih_\iota}\right)^\top\left(\UA^\top\UB-\frac{\ubar_\iota^\top\ut_\iota}{\pih_\iota}\right)\\
  &= \UB^\top\UA\UA^\top\UB + \frac{\ut_\iota^\top\ubar_\iota\ubar_\iota^\top\ut_\iota}{\pih_\iota^2} - 2 \frac{\ut_\iota^\top\ubar_\iota}{\pih_\iota}\UA^\top\UB
\end{align*}
where in the second equality we used the fact that
$$ \UB^\top\UA\cdot\frac{\ubar_\iota^\top\ut_\iota}{\pih_\iota} = \frac{\ut_\iota^\top\ubar_\iota}{\pih_\iota}\cdot\UA^\top\UB \ . $$
Let $\Upsilon=\UB^\top\UA\UA^\top\UB$, and $\Eb_\iota=\E\big[\Xb_\iota^\top\Xb_\iota\big]$. It follows that
\begin{align*}
  \Eb_\iota &= \Upsilon + \E\left[\frac{\ut_\iota^\top\ubar_\iota\ubar_\iota^\top\ut_\iota}{\pih_\iota^2}\right] - 2\cdot\E\left[\frac{\ut_\iota^\top\ubar_\iota}{\pih_\iota}\right]\cdot\UA^\top\UB\\
  &= \Upsilon + \E\left[\frac{\ellb_\iota}{\pih_\iota^2}\cdot\ut_\iota^\top\ut_\iota\right] - 2\cdot\left(\sum_{l=1}^N\pih_l\cdot\frac{\ut_l^\top\ubar_l}{\pih_l}\right)\cdot\UA^\top\UB\\
  &= \Upsilon + \left(\sum_{l=1}^N\pih_l\cdot\frac{\ellb_l}{\pih_l^2}\cdot\ut_l^\top\ut_l\right)-2\cdot\UB^\top\UA\cdot\UA^\top\UB\\
  &= \Upsilon + \left(\sum_{l=1}^N\frac{\ellb_l}{\pih_l}\cdot\ut_l^\top\ut_l\right) - 2\Upsilon\\
  &\leqslant \left(\sum_{l=1}^N\frac{\ellb_l}{\beta\cdot\pi_l}\cdot\ut_l^\top\ut_l\right) - \Upsilon\\
  &= \left(\sum_{l=1}^N\frac{\ellb_l}{\beta\cdot\sqrt{\ellb_l\cdot\ellt_l}\big/\tau}\cdot\ut_l^\top\ut_l\right) - \Upsilon\\
  &= \frac{\tau}{\beta}\cdot\left(\sum_{l=1}^N\sqrt{\ellb_l/\ellt_l}\cdot\ut_l^\top\ut_l\right) - \Upsilon\\
  &= \frac{\tau}{\beta}\cdot\Phit - \Upsilon\ .
\end{align*}
By the triangle inequality, and the fact that $\|\UA\|_2=\|\UB\|_2=1$, we conclude that
\begin{align*}
  \|\Eb_\iota\|_2 &= \left\|\frac{\tau}{\beta}\cdot\Phit - \Upsilon\right\|_2 \\
  &\leqslant \frac{\tau}{\beta}\cdot\|\Phit\|_2 + \|\Upsilon\|_2\\
  &\leqslant \frac{\tau}{\beta}\cdot\|\Phit\|_2 + 1\\
  &= \frac{\tau}{\beta}\cdot\lambda_1(\Phit) + 1\ .
\end{align*}
An analogous calculation yields that $\big\|\E\big[\Xb_\iota\Xb_\iota^\top\big]\big\|_2\leqslant \frac{\tau}{\beta}\cdot\lambda_1(\Phib) + 1$.
\end{proof}

\begin{Lemma}
\label{lem_3}
\textit{Considering any $\Ab\in\R^{N\times L}$ and $\Bb\in\R^{N\times M}$, we have: $\lambda_1(\Phit)\leqslant\max_{l\in\N_N}\{\alpha_l\}\cdot\max_{i\in\N_N}\{\ellt_i\}$ and $\lambda_1(\Phib)\leqslant\max_{l\in\N_N}\{1/\alpha_l\}\cdot\max_{i\in\N_N}\{\ellb_i\}$.}
\end{Lemma}

\begin{proof}
Recall that $\lambda_1(\Mb)=\max_{\xb\in\Sbb^{n-1}}\left\{\xb^\top\Mb\xb\right\}$, for $\Mb\in\R^{n\times n}$ and $\Sbb^{n-1}$ the $n-1$ dimensional $\ell_2$-sphere, \textit{i.e.} $\Sbb^{n-1}=\{\xb\in\R^n:\|\xb\|_2=1\}$. It follows that
\begin{align*}
  \lambda_1(\Phit) &= \max_{\xb\in\Sbb^{M-1}}\left\{\xb^\top\left(\sum_{j=1}^N\alpha_j\cdot\ut_j^\top\ut_j\right)\xb\right\}\\
  &\leqslant \max_{l\in\N_N}\{\alpha_l\}\cdot\max_{\xb\in\Sbb^{M-1}}\left\{\xb^\top\left(\sum_{j=1}^N\ut_j^\top\ut_j\right)\xb\right\}\\
  &= \max_{l\in\N_N}\{\alpha_l\}\cdot\max_{\xb\in\Sbb^{M-1},\ i\in\N_N}\left\{\langle\xb,\ut_i\rangle^2\right\}\\
  &= \max_{l\in\N_N}\{\alpha_l\}\cdot\max_{i\in\N_N}\big\{\ellt_i\big\}\ ,
\end{align*}
which proves the first claim. The bound on $\lambda_1(\Phib)$ is proven similarly.
\end{proof}

$\ind$ We now have everything we need in order to prove our main result regarding Algorithm \ref{alg_jlvg}. By combining Lemmas \ref{lem_2} and \ref{lem_3}, we get that
\begin{align*}
\sigma^2 &= \max\left\{\big\|\E\big[\Xb^\top\Xb\big]\big\|_2,\big\|\E\big[\Xb\Xb^\top\big]\big\|_2\right\}\\
&\leqslant \frac{\tau}{\beta}\cdot \max\left\{\max_{l\in\N_N}\{\alpha_l\}\cdot\max_{i\in\N_N}\{\ellt_i\} , \max_{l\in\N_N}\{1/\alpha_l\}\cdot\max_{i\in\N_N}\{\ellb_i\}\right\}+1\ .
\end{align*}
We can carry over this bound on the variance in our analysis, though the resulting expressions are not helpful to understanding the value of our method, as the variance is also an artifact of the matrices $\Ab$ and $\Bb$. To simplify our statement and presentation, we therefore assume that $\Ab$ and $\Bb$ satisfy
\begin{equation}
\label{ass_variance}
  \max\left\{\max_{l\in\N_N}\{\alpha_l\}\cdot\max_{i\in\N_N}\{\ellt_i\} , \max_{l\in\N_N}\{1/\alpha_l\}\cdot\max_{i\in\N_N}\{\ellb_i\}\right\} \leqslant \frac{\mu}{\tau}\ ,
\end{equation}
which implies that $\sigma^2\leqslant\mu/\beta+1$. It is also evident in our results that the relative sizes $L$ and $M$ also play a role, something which does not occur when $\Ab=\Bb$ nor when considering $\ell_2$-subspace embeddings.

\begin{Thm}
\label{general_thm}
\textit{Let $\delta>0$, and assume that \eqref{ass_variance} holds. Then, the sketching matrix $\Pibh$ of Algorithm \ref{alg_jlvg} with sampling according to $\{\pih_i\}_{i=1}^N$ meets condition \eqref{sp_AMM_char} with probability at least $\big(1-\delta\cdot\max\big\{\sqrt{L/M},\sqrt{M/L}\big\}\big)$, for $r=\Theta\big(\mu\log(2\mu/\delta)/(\beta\epsilon^2)\big)$.}
\end{Thm}

\begin{proof}
Under the assumption that \eqref{ass_variance} holds for the pair of matrices $\Ab$ and $\Bb$, we know that $\sigma^2\leqslant\mu/\beta+1$ for all random sampling trials of Algorithm \ref{alg_jlvg}. The corresponding argument $\Dbv_{\Pibh}$ of \eqref{sp_AMM_char} for $r\geqslant1$, is equal to $\frac{1}{r}\sum_{l=1}^r\Xb_l$ for $r$ realizations of \eqref{def_mrv}. By substituting $r=\Theta(\mu\log(2\mu/\delta)/(\beta\epsilon^2))$, and the bounds $\sigma^2\leqslant\mu/\beta+1$, $\gamma\leqslant\mu/\beta+1$ in the exponent of the bound in Theorem \ref{Bern_ineq}, we get
\begin{align*}
  \frac{-r\epsilon^2/2}{\mu/\beta+1+(\mu/\beta+1)\epsilon/3} &= \frac{-r\beta\epsilon^2/2}{(\mu+\beta)(1+\epsilon/3)}\\
  &= \frac{-\Theta\left(\mu\log(2\mu/\delta)/(\beta\epsilon^2)\right)\frac{\beta\epsilon^2}{2(1+\epsilon/3)}}{\mu+\beta}\\
  &= -\Theta(\log(2\mu/\delta))\ .
\end{align*}
Assume for now that $L\geqslant M$. It follows that
\begin{align*}
  \Pr\big[\|\Dbv_{\Pibh}\|_2>\epsilon\big] &\leqslant (L+M)\cdot e^{-\Theta(\log(2\mu/\delta))}\\
  &\leqslant 2L\cdot e^{-\Theta(\log(2\mu/\delta))}\\
  &= e^{\log(2L)-\Theta(\log(2\mu/\delta))}\\
  &= e^{\Theta\left(\log\left(\frac{L\delta}{\mu}\right)\right)}\\
  &= e^{\Theta\big(\log \big(\delta\cdot\sqrt{L/M}\big)\big)}\\
  &= \delta\cdot\sqrt{L/M}\ .
\end{align*}
By taking the complementary event, we deduce that
\begin{equation}
\label{bd_1}
  \Pr\left[\big\|\UA^\top\UB - (\Pibh\UA)^\top(\Pibh\UB)\big\|_2\leqslant\epsilon\right] > 1-\delta\cdot\sqrt{L/M}\ .
\end{equation}
Now, under the assumption that $L\leqslant M$, the same calculation yields
\begin{equation}
\label{bd_2}
  \Pr\left[\big\|\UA^\top\UB - (\Pibh\UA)^\top(\Pibh\UB)\big\|_2\leqslant\epsilon\right] > 1-\delta\cdot\sqrt{M/L}\ .
\end{equation}
By combining \eqref{bd_1} and \eqref{bd_2}, we complete the proof.
\end{proof}

$\ind$ As was briefly discussed in Subsection \ref{subsec_jls}, the only difference in sampling according to approximate joint leverage scores $\{\pih_i\}_{i=1}^N$ and the exact scores $\{\pi_i\}_{i=1}^N$, is that with the former we need to oversample by a factor of $1/\beta$. We see this explicitly through their guarantees in terms of the error $\errAB^{\ell_2,\Pibh}$ defined in Lemma \ref{lem_error}, provided in Proposition \ref{prop_main_thm} and Corollary \ref{cor_main_thm} respectively.

\begin{Prop}
\label{prop_main_thm}
\textit{For $\epstil,\delta>0$, $L\geqslant M$, and $r=\Theta\big(\mu\log(2\mu/\delta)\cdot\big(\sigma_1(\Ab)\sigma_1(\Bb)/(\beta\epstil)\big)^2\big)$, the sketching matrix $\Pibh$ of Algorithm \ref{alg_jlvg} with sampling according to $\{\pih_i\}_{i=1}^N$, satisfies
$$ \Pr\left[\errAB^{\ell_2,\Pibh}\leqslant\epstil\right]\geqslant1-\delta\cdot\sqrt{L/M} .$$
}
\end{Prop}

\begin{proof}
Let $\epstil=\epsilon\cdot\sigma_1(\Ab)\sigma_1(\Bb)$. By Lemma \ref{lem_error} we know that $\errAB^{\ell_2,\Pibh}\leqslant\epsilon\cdot\sigma_1(\Ab)\sigma_1(\Bb)$, so by our choice of $\epstil$ we have $\Pr\left[\errAB^{\ell_2,\Pibh}\leqslant\epstil\right]=\Pr\left[\big\|\Delta^\Pibh_{\Ab,\Bb}\big\|\leqslant\epsilon\right]$. By Theorem \ref{general_thm}; for the case where $L\geqslant M$, it follows that $\Pr\big[\big\|\Delta^\Pibh_{\Ab,\Bb}\big\|\leqslant\epsilon\big]\leqslant1-\delta\cdot\sqrt{\frac{L}{M}}$, for $r=\Theta\big(\mu\log(2\mu/\delta)/\epsilon^2\big)=\Theta\big(\mu\log(2\mu/\delta)\cdot(\sigma_1(\Ab)\sigma_1(\Bb)/\epstil)^2\big)$.
\end{proof}

\begin{Cor}
\textit{\label{cor_main_thm}
Let $\epstil,\delta>0$, $L\geqslant M$, $r=\Theta\big(\mu\log(2\mu/\delta)\cdot\big(\sigma_1(\Ab)\sigma_1(\Bb)/\epstil\big)^2\big)$, and consider sampling according to the joint leverage scores distribution $\{\pi_i\}_{i=1}^N$; for which Algorithm \ref{alg_jlvg} produces the sketching matrix $\Pib$. We then have
$$ \Pr\left[\errAB^{\ell_2,\Pib}\leqslant\epstil\right]\geqslant1-\delta\cdot\sqrt{L/M} .$$
}
\end{Cor}

\begin{proof}
If we sample according to the exact distribution $\{\pi_i\}_{i=1}^N$, it follows that $\beta=1$. By Proposition \ref{prop_main_thm}, the statement is immediate.
\end{proof}

$\ind$ A drawback of the above results is that the largest singular value of both matrices appear in the sampling complexities. Note though that if we made an analogous assumption to Theorem \ref{sp_bd_thm}; that $\|\Ab\|_2\leqslant1$, this would not be a concern.\\
$\ind$ To make a fair comparison with the $\crmm$ guarantee of Theorem \ref{sp_bd_thm}, let us assume that $\|\Ab\|_2,\|\Bb\|_2\leqslant1$, and that the sampling complexities for the two algorithms are the same. For Algorithm \ref{alg_jlvg} we would then have $r=\Theta\big(\mu\log(2\mu/\delta')/\epsilon^2\big)$ for $\delta'>0$, $\mu=96\|\Ab\|_F\|\Bb\|_F$, and we would get the same probability of error when the accuracy parameters are selected such that $\delta'=\Theta\big(\epsilon^2\sqrt{\delta}\big)$. Moreover, what we presented holds true for the product of any pair of distinct matrices $\Ab$ and $\Bb$, while Theorem \ref{sp_bd_thm} only considers the Gram matrix of $\Ab$.\\
$\ind$ We conclude this section by restating that the result presented in Theorem \ref{general_thm}, is a more general result of the guarantee obtained by the leverage score sampling sketch for a $\ell_2$-subspace embedding. Specifically, by considering $\Bb=\Ab$, we recover the following theorem as a corollary.

\begin{Thm}[\cite{DMMS11,Woo14}]
\label{lvg_score_sampling_thm}
\textit{The leverage score sampling sketch $\Sbt\in\R^{r\times N}$ with approximate sampling distribution according to $\{\ellt_i\}_{i=1}^N$ for which $\ellt_i\geqslant \beta\ell_i$ of matrix $\Ab\in\R^{N\times d}$ and $\beta\in(0,1]$, then $\Sbt$ is a $\ell_2$-s.e. of $\Ab$. Specifically, for fixed $\epsilon,\delta>0$ and $\delta$, $\Ubv$ a reduced orthonormal basis of $\Ab$, and $r=\Theta\big(\log(2/\delta)/(\beta\epsilon^2)\big)$, $\Sbt$ satisfies \eqref{se_char} with probability at least $1-\delta$.}
\end{Thm}

\section{Implications on other Approximation Algorithms}
\label{jls_other}

$\ind$ In this section, we discuss how the analysis of the proposed AMM through joint leverage scores can be used to derive spectral guarantees for other similar AMM algorithms, and show how it can be utilized for spectral sparsification of `intersection graphs'. In Subsection \ref{CR_subsec} we derive a spectral guarantee for $\crmm$, and in Subsection \ref{SRHT_subsec} we briefly discuss how our proposed AMM can be utilized by first applying a random projection and then performing uniform random sampling on the transformed matrices' row pairs. In Subsection \ref{graph_spars_subsec}, we discuss the connection to graph sparsification.

\subsection{$\crmm$ through Approximate Joint Leverage Scores}
\label{CR_subsec}

$\ind$ A pitfall of our algorithm, is that it requires knowledge of a singular value or $\QR$ decomposition for each $\Ab$ and $\Bb$, in order to calculate the joint leverage scores. We therefore need $O(N^2(L+M))$ operations to compute $\{\pi_i\}_{i=1}^N$. Instead, more efficient leverage score approximation algorithms could be used \cite{DMMW12,RCCR18,SS11}, that do not require directly computing a left orthonormal basis of each matrix. This is part of the reason Theorem \ref{general_thm} and Proposition \ref{prop_main_thm} are presented in terms of an approximate distribution $\{\pih_i\}_{i=1}^N$. The sampling procedure of the $\crmm$ algorithm on the other hand, can be done efficiently with only $O(1)$ additional storage space, by the pass-efficient $\textup{\textsf{SELECT}}$ algorithm \cite[page 13]{Mah16}.\\
$\ind$ In Proposition \ref{sp_guar_CR}, we quantify an AMM spectral bound for $\crmm$, by using the distribution $\{p_i\}_{i=1}^N$ of $\crmm$ as a surrogate to $\{\pi_i\}_{i=1}^N$. In many practical cases, we empirically observed that the two distributions were relatively close to each other, \textit{i.e.} $p_i\geqslant\betat\pi_i$ for a $\betat\in(0,1]$ close to $1$. What Lemma \ref{lemma_dist} says is that the oversampling parameter when performing Algorithm \ref{alg_jlvg} with the sampling distribution of $\crmm$ \eqref{CR_distr}, is $\betat=\sigma_{\min}(\Ab)\cdot\sigma_{\min}(\Bb)\cdot\tau/\phi$.

\begin{Lemma}
\label{lemma_dist}  
\textit{For all $i\in\N_N$, we have $\pi_i\cdot\big(\sigma_{\min}(\Ab)\cdot\sigma_{\min}(\Bb)\cdot\tau/\phi\big)\leqslant p_i$.}
\end{Lemma}

\begin{proof}
Recall that $p_i=\|\Ab_{(i)}\|_2\cdot\|\Bb_{(i)}\|_2 \big/ \phi$; for $\phi=\sum_{j=1}^N\|\Ab_{(j)}\|_2\cdot\|\Bb_{(j)}\|_2$. Consider the reduced $\QR$ factorization of $\Ab=\Qb\Rb$, where $\Qb\in\R^{N\times L}$ matches $\UA$, and $\Rb$ is an upper triangular matrix. It follows that $\Qb=\Ab\Rb^{-1}$ for $\Rb^{-1}$ also an upper triangular matrix, and
$$ \sqrt{\ellb_i} = \|\Qb_{(i)}\|_2 = \|\Ab_{(i)}\cdot\Rb^{-1}\|_2 \leqslant \|\Ab_{(i)}\|_2\cdot\|\Rb^{-1}\|_2 = \|\Ab_{(i)}\|_2 \cdot\sigma_{\max}(\Rb^{-1}) = \|\Ab_{(i)}\|_2 \cdot \frac{1}{\sigma_{\min}(\Ab)}\ . $$
Similarly, one can show that $\sqrt{\ellt_i}\leqslant \|\Bb_{(i)}\|_2 \cdot \frac{1}{\sigma_{\min}(\Bb)}$. By combining the two bounds, we get
$$ \js_{i,i}\cdot\sigma_{\min}(\Ab)\cdot\sigma_{\min}(\Bb) \leqslant \|\Ab_{(i)}\|_2\cdot\|\Bb_{(i)}\|_2 = p_i\cdot\phi \quad \implies \quad \pi_i\cdot\big(\sigma_{\min}(\Ab)\cdot\sigma_{\min}(\Bb)\cdot\tau/\phi\big)\leqslant p_i $$
which completes the proof.
\end{proof}

\begin{Prop}
\label{sp_guar_CR}
\textit{Let $L>M$ and $\delta,\epsilon>0$. For $\betat=\big(\sigma_{\min}(\Ab)\cdot\sigma_{\min}(\Bb)\cdot\tau/\phi\big)$, $r=\Theta(\mu\log(2\mu/\delta)/(\betat\epsilon^2))$, the resulting sketching matrix $\Sbt$ of $\crmm$ satisfies
$$ \Pr\left[\|\UA^\top \UB - (\Sbt\UA)^\top(\Sbt\UB)\|_2\leqslant\epsilon\right]\geqslant 1-\sqrt{L/M}\cdot\delta\ . $$
}
\end{Prop}

\begin{proof}
This follows directly by taking $\betat=\big(\sigma_{\min}(\Ab)\cdot\sigma_{\min}(\Bb)\cdot\tau/\phi\big)$; based on the conclusion of Lemma \ref{lemma_dist}, and using $\{p_i\}_{i=1}^N$ as an approximate distribution to $\{\pi_i\}_{i=1}^N$ in Theorem \ref{general_thm}. Furthermore, the constant $\betat$ lies in $(0,1]$, as all its terms are positive; and $\{p_i\}_{i=1}^N$, $\{\pi_i\}_{i=1}^N$ are both valid sampling distributions.
\end{proof}

\subsection{Data-Oblivious AMM through Joint Leverage Scores}
\label{SRHT_subsec}

$\ind$ Recall that the `\textit{Subsampled Randomized Hadamard Transform}' (SRHT) \cite{DMMS11} is comprised of three matrices: $\Omb\in\R^{r\times N}$ a uniform sampling (with replacement) and rescaling matrix of $r$ rows, $\Hbh_N\in\{\rpm1/\sqrt{N}\}^{N\times N}$ the normalized Hadamard matrix for $N=2^n$, and $\Db\in\{0,\rpm1\}^{N\times N}$ with i.i.d. diagonal Rademacher random entries. The structured sketching matrix $\Sbh\coloneqq\Omb\Hbh_N\Db\in\R^{r\times N}$ of a matrix $\Ab\in\R^{N\times L}$, flattens the leverage scores of $\Ab$ before performing sampling through $\Omb$, \textit{i.e.} the leverage scores of $\Abwt\coloneqq\Hbh_N\Db\cdot\Ab$ are approximately equal \cite{AC06,DMMS11}.\footnote{For example, see \cite[Lemma 2.1]{AC06}, \cite[Lemma 3]{DMMS11}, \cite[Lemma 6]{CMPH23}. In our statement of Corollary \ref{cor_data_obl} we do not precisely state the accuracy of how close the flattened normalized scores are to the exact uniform distribution, as we consider any such random matrix. Precise statements can be made for the aforementioned random projections, which ``uniformize'' the leverage scores.} Hence, once we apply the randomized transform $\Hbh_N\Db$ to $\Ab$, we can sample rows of $\Abwt$ uniformly at random with replacement (which is done through $\Omb$) without directly utilizing any information from $\Ab$; to get an \textit{oblivious} $\ell_2$-subspace embedding sketch $\Abh\coloneqq\Sbh\Ab$.\\
$\ind$ Now, consider two distinct matrices $\Ab\in\R^{N\times L}$ and $\Bb\in\R^{N\times M}$. By applying to each of them the above randomized transform we get $\Abwt=\Hbh_N\Db\cdot\Ab$ and $\Bbwt=\Hbh_N\Db\cdot\Bb$, both of which have \textit{approximately} uniform leverage scores. Hence, the joint leverage scores distribution of the pair $(\Abwt,\Bbwt)$ is also approximately uniform. From this, the conclusion of Theorem \ref{general_thm} directly applies to the approximate product
$$ \Ab^\top\cdot\Bb \approx \overbrace{\left(\Hbh_N\Db\cdot\Ab\right)^\top}^{\Abwt^\top} \cdot \left(\Omb^\top\Omb\right) \cdot \overbrace{\left(\Hbh_N\Db\cdot\Bbh\right)}^{\Bbwt} = \Abh^\top\cdot\Bbh $$
where for $i(j)\in\N_N$ the index of the sampled pair $(\Abwt_{i(j)},\Bbwt_{i(j)})$:
\begin{equation}
\label{omega_un_sampling}
  \Omb_{j,i(j)} = 1\big/\sqrt{r\cdot(1/N)} = \sqrt{N/r}
\end{equation}
and all other entries of $\Omb$ are zero.\\
$\ind$ Furthermore, random Gaussian or orthonormal matrices could also be considered in place of the transform $\Hbh_N\Db$, as these also flatten the leverage scores of the matrices; and additionally provide security of the information stored in $\Ab$ and $\Bb$ \cite{CMPH22,ZWL08}. The main drawback of these approaches, is that they are slower than the SRHT; which can compute the sketch $\Abh$ in $O(NL\log N)$ time through Fourier methods.

\begin{Cor}
\label{cor_data_obl}
\textit{Let $\Fbv\in\R^{N\times N}$ be a random matrix which with high probability flattens the leverage scores of $\Ab$ and $\Bb$, \textit{i.e.} the leverage scores of $\Abwt\coloneqq\Fbv\cdot\Ab$ and $\Bbwt\coloneqq\Fbv\cdot\Bb$ are approximately uniform. Then, after applying uniform sampling on the row vector pairs of $(\Abwt,\Bbwt)$ with rescaling according to $\Omb\in\R^{r\times N}$ from \eqref{omega_un_sampling} in Algorithm \ref{alg_jlvg}, we obtain an estimator 
$$ \widetilde{\Ybv}\coloneqq\big(\Omb\Abwt\big)^\top\cdot\big(\Omb\Bbwt\big)\approx\Ab^\top\cdot\Bb $$
which satisfies \eqref{sp_AMM_char} with overwhelming probability.}
\end{Cor}

\begin{proof}
By the assumption on $\Fbv$, we know that with high probability the leverage scores distributions $\big\{\pi^{\Abwt}_i\big\}_{i=1}^N$ of $\Abwt$ and $\big\{\pi^{\Bbwt}_i\big\}_{i=1}^N$ of $\Bbwt$ are both approximately uniform, \textit{i.e.}
$$ \pi^{\Abwt}_j \approx \pi^{\Abwt}_l \approx 1/N \qquad \text{and} \qquad \pi^{\Bbwt}_j \approx \pi^{\Bbwt}_l \approx 1/N $$ 
for every pair $(j,l)\in\N_N\times\N_N$. Consequently, with high probability the joint leverage scores distribution
$$ \left\{\piwt_i\coloneqq\big({\pi^{\Abwt}_i\cdot\pi^{\Bbwt}_i}\big)^{1/2}\right\}_{i=1}^N $$
of the pair $(\Abwt,\Bbwt)$ is also close to uniform. Let $\widetilde{\beta}\in(0,1]$ be the misestimation factor which quantifies how far from uniform is the distribution $\{\piwt_i\}_{i=1}^N$. By directly applying Theorem \ref{general_thm} to the pair $(\Abwt,\Bbwt)$ with $r=\Theta\big(\mu\log(2\mu/\delta)/(\widetilde{\beta}\epsilon^2)\big)$ sampling trials, we complete the proof.
\end{proof}

\subsection{Graph Spectral Sparsifiers}
\label{graph_spars_subsec}

$\ind$ A seemingly unrelated object, which has many applications in data science, is the graph Laplacian. Graphs are often used to model phenomena in which there is a flow around a network with high and low conductance regions; \textit{e.g.} electrical circuits, and are prevalent in many domains of modern computer science and engineering. Due to the large size of the networks considered, a sparser version is preferred for its more efficient storage and processing requirements. This gives rise to the notion of a \textit{graph spectral sparsifier} \cite{ST10,SS11,CH23}, which requires that the Laplacian quadratic form of the sparsifier approximates that of the original. In \cite{CH23}, it was shown how $\crmm$ can be utilized to produce minimum variance unbiased \textit{additive} and \textit{multiplicative} spectral sparsifiers. The AMM algorithm proposed in this paper can be viewed as a generalization of the well-known \textit{effective resistance} sampling approach, which is the state-of-the-art in spectral sparsification through edge sampling \cite{SS11}.\\
$\ind$ The state-of-the-art approach to spectral sparsification of a weighted graph $G=(V,E,w)$ is to sample edges according to `\textit{effective resistances}', which has connections to the leverage scores of the graph's \textit{boundary matrix} $\Bvb\in\R_{\geqslant0}^{|E|\times N}$, for $|V|=N$ \cite{DM10,Spi10,SS11}. This approach leads to a nearly-linear time algorithm that produces high-quality sparsifiers, by approximating $G$'s \textit{Laplacian} $\Lb$ which is defined as the outer-product of its boundary matrix; \textit{i.e.} $\Lb=\Bvb^\top\Bvb$. That is, we get a sparser graph $\Gt=(V,\tilde{E},\tilde{w})$ with Laplacian $\Lbt\approx\Lb$.

\begin{Def}
\textit{An \textbf{\textit{$\varepsilon$-spectral sparsifier}} for $\varepsilon\in(0,1)$ of $G$ with Laplacian $\Lb$, is a sketched graph $\Gt$ whose Laplacian $\Lbt$ for all $\xb\in\R^N$, satisfies with high probability}
\begin{equation}
\label{sp_spars_id}
  (1-\varepsilon)\xb^\top\Lbt\xb \leqslant \xb^\top\Lb\xb \leqslant (1+\varepsilon)\xb^\top\Lbt\xb\ .
\end{equation}
\end{Def}

$\ind$ The benefit of using the bound of Theorem \ref{general_thm}, is that it can be applied to directly approximate the \textit{intersection} of two different graphs on $V$; say $G_1=\left(V,E_1,w^1\right)$ and $G_2=\left(V,E_2,w^2\right)$, which has found applications in areas such as image segmentation. The intersection of these two graphs is defined as $H=\left(V,E_H,w^H\right)$, where $E_H=E_1\cap E_2$ and $w_e^H=\sqrt{w_e^1\cdot w_e^2}$ for each $e\in E_H$ (w.l.o.g., we assume that the indices of the edges are the same in $G_1$ and $G_2$).\\
$\ind$ We can use AMM through joint leverage scores to directly approximate $\Lb_H=\Bvb_1^\top\Bvb_2$, for $\Bvb_1,\Bvb_2$ the boundary matrices of the two graphs.\footnote{In this setting, we consider all possible ${N\choose2}$ edges $E\times E$, with an arbitrary but fixed ordering. We then assume that $\Bvb_\iota\in\R^{{N\choose2}\times N}$ for each $\iota\in\{1,2\}$, where the $i^{th}$ row of $\Bvb_\iota$ is the all zeros vector; if the edge corresponding to this row is not in $E_\iota$.} This results in the Laplacian $\Lb_{\tilde{H}}$ of a sparsified version of $H$. The techniques of \cite{SS11} on sampling according to effective resistances do not directly apply to the intersection $H$ of graphs $G_1, G_2$ without explicitly extracting information from $H$, though through what was proposed in this paper; one can extend the work of \cite{SS11} to consider `\textit{joint effective resistances}' defined across the two graphs. The main benefit is that we can sparsify $H$, without constructing it explicitly.\\
$\ind$ The orientation of an edge $e=(u,v)$, is represented in the \textit{incidence vector} $\chi_e=\vec{\eb}_u-\vec{\eb}_v$; for $\{\vec{\eb}_i\}_{i=1}^{N}$ the standard basis vectors of $\R^{N}$. Let $\xbvt_e=\sqrt{\Lb^{\dagger}}\chi_e$, for each $e\in E$. Then, the effective resistance of $e$ is defined as $r_e\coloneqq\|\xbvt_e\|_2^2$. Denote the effective resistances of $G_\iota$ and the corresponding individual leverage scores of $\Bvb_\iota$ by $\big\{r_e^\iota\big\}_{e\in E_\iota}$ and $\{\ell_e^\iota\}_{e\in E_\iota}$ respectively, for each $\iota\in\{1,2\}$. Considering the aforementioned distinct matrices $G_1$ and $G_2$, we can define their \textit{joint effective resistances} as
$$ r_e^H \coloneqq \frac{\sqrt{\ell_e^1\cdot\ell_e^2}}{\sqrt{w_e^1\cdot w_e^2}} = \sqrt{r_e^1\cdot r_e^2} $$
for each of $e\in E_H$.\\
$\ind$ According to the $\textup{\textsf{Sparsify}}$ algorithm proposed in \cite{SS11}; in order to obtain a sparse graph $\Ht=\big(V,E_{\Ht},w^{\Ht}\big)$ of $H$, we sample edges from $E_H$ with replacement according to the distribution
\begin{equation*}
\label{sampling_distr_ER}
  p_e = \frac{w^H_e\cdot r^H_e}{\sum_{e'\in E_H}w_{e'}^H\cdot r^H_{e'}}\ ,
\end{equation*}
for a total of $r$ random sampling trials. Each time $e'\in E_H$ is sampled, its weight $w^{\Ht}_e$ in the new sparsified graph $\Ht$ is increased by $w^H_{e'}\big/(rp_{e'})$. The physical interpretation of the increase in an edge's weight, is an increase in its \textit{conductance}.

\begin{Cor}
\label{main_thm}
\textit{Suppose $H$ is the intersection of graphs $G_1$ and $G_2$ on $V$; for which $|V|=N$, and consider an $\epsilon\in(0,1)$ that satisfies $1\big/\sqrt{N}<\varepsilon<1$. Then, the sampling procedure described above with $r=O\big(\frac{N\log N}{\epsilon^2}\big)$ trials, produces an $\varepsilon$-spectral sparsifier $\Ht$, of the intersection graph $H$. Moreover, we have $\E\left[|E_{\Ht}|\right]=O\big(\frac{N\log N}{\epsilon^2}\big)$.}
\end{Cor}

\begin{proof}
The proof is a direct application of the results in \cite{SS11} to the intersection graph $H$, whose Laplacian is $\Lb_H=\Bvb_1^\top\Bvb_2$. By \cite[Theorem 1]{SS11}, \cite[Lemma 4]{SS11} and what was defined in this subsection, the fact that $\Ht$ is an $\varepsilon$-spectral sparsifier of $H$ is immediate. The expectation on the number of edges retained in $\Ht$ follows from \cite[Theorem 12]{SS11}.
\end{proof}

\section{Concluding Remarks and Future Work}
\label{concl_sec}

$\ind$ The main contributions of this paper are the following. First, we connected ideas from two distinct RandNLA algorithms to $\mathrm{(i)}$ define a new generalized notion of leverage scores; which when $\Bb=\Ab$ results in the leverage scores of $\Ab$, and $\mathrm{(ii)}$ developed an AMM algorithm for which a spectral guarantee was provided. Our algorithm produces a low-rank approximation, which is more efficient to work with in terms of storage and computation. In order to provide a guarantee for our algorithm, we proposed a condition for AMM which is similar to the $\ell_2$-subspace embedding property. We also gave a connection to graph sparsifiers, and presented a generalized notion of effective resistances, which is worth investigating further.\\
$\ind$ Another area which could be benefit from joint leverage score sampling, is that of tensor products and decompositions; where analogous notions could potentially be defined. It would be interesting to see what other optimization problems and applications can benefit from utilizing sampling according to the joint leverage scores distribution. We believe our technique could be incorporated in a distributed optimization setting for least squares regression, in which there are two parties who have distinct features of a dataset, but never share information with each other. Furthermore, the study of leverage scores through information theory and how they relate to the notion of entropy, is worthwhile future work. This may provide insight to what other applications and approximation algorithms could benefit from similar sampling techniques.



\bibliographystyle{alpha}
\bibliography{refs.bib}

@inproceedings{CPH20a,
  title={{Weighted Gradient Coding with Leverage Score Sampling}},
  author={Charalambides, Neophytos and Pilanci, Mert and Hero, Alfred},
  booktitle={ICASSP 2020-2020 IEEE International Conference on Acoustics, Speech and Signal Processing (ICASSP)},
  pages={5215--5219},
  year={2020},
  organization={IEEE}
}

@INPROCEEDINGS{CPH20b,
  author={Charalambides, Neophytos and Pilanci, Mert and Hero, Alfred},
  booktitle={ICASSP 2021 - 2021 IEEE International Conference on Acoustics, Speech and Signal Processing (ICASSP)}, 
  title={{Approximate Weighted $CR$ Coded Matrix Multiplication}},
  year={2021},
  volume={},
  number={},
  pages={5095-5099},
  doi={10.1109/ICASSP39728.2021.9413800}
}

@article{Woo14,
  title={{Sketching as a Tool for Numerical Linear Algebra}},
  author={Woodruff, David P.},
  journal={Foundations and Trends{\textregistered} in Theoretical Computer Science},
  volume={10},
  number={1--2},
  pages={1--157},
  year={2014},
  publisher={Now Publishers, Inc.}
}

@article{Mah16,
  title={{Lecture Notes on Randomized Linear Algebra}},
  author={Mahoney, Michael W.},
  journal={arXiv preprint arXiv:1608.04481},
  year={2016}
}

@article{DMMW12,
  title={Fast approximation of matrix coherence and statistical leverage},
  author={Drineas, Petros and Magdon-Ismail, Malik and Mahoney, Michael W. and Woodruff, David P.},
  journal={Journal of Machine Learning Research},
  volume={13},
  number={Dec},
  pages={3475--3506},
  year={2012}
}

@inproceedings{RCCR18,
  title={On fast leverage score sampling and optimal learning},
  author={Rudi, Alessandro and Calandriello, Daniele and Carratino, Luigi and Rosasco, Lorenzo},
  booktitle={Advances in Neural Information Processing Systems},
  pages={5672--5682},
  year={2018}
}

@article{ST10,
  title={{Spectral Sparsification of Graphs}},
  author={Spielman, Daniel A. and Teng, Shang-Hua},
  journal={SIAM Journal on Computing},
  volume={40},
  number={4},
  pages={981--1025},
  year={2011},
  publisher={Society for Industrial and Applied Mathematics Philadelphia, PA, USA}
}

@article{SS11,
  title={{Graph Sparsification by Effective Resistances}},
  author={Spielman, Daniel A. and Srivastava, Nikhil},
  journal={SIAM Journal on Computing},
  volume={40},
  number={6},
  pages={1913--1926},
  year={2011},
  publisher={SIAM}
}

@article{DM10,
  title={{Effective Resistances, Statistical Leverage, and Applications to Linear Equation Solving}},
  author={Drineas, Petros and Mahoney, Michael W.},
  journal={arXiv preprint arXiv:1005.3097},
  year={2010}
}

@inproceedings{DK01,
  title={{Fast Monte-Carlo Algorithms for Approximate Matrix Multiplication}},
  author={Drineas, Petros and Kannan, Ravi},
  booktitle={Proceedings of the 42nd IEEE symposium on Foundations of Computer Science},
  pages={452--459},
  year={2001}
}

@article{DKM06a,
  title={{Fast Monte Carlo Algorithms for Matrices I: Approximating Matrix Multiplication}},
  author={Drineas, Petros and Kannan, Ravi and Mahoney, Michael W.},
  journal={SIAM Journal on Computing},
  volume={36},
  number={1},
  pages={132--157},
  year={2006},
  publisher={SIAM}
}

@article{DKM06b,
  title={{Fast Monte Carlo Algorithms for Matrices II: Computing a Low-Rank Approximation to a Matrix}},
  author={Drineas, Petros and Kannan, Ravi and Mahoney, Michael W.},
  journal={SIAM Journal on computing},
  volume={36},
  number={1},
  pages={158--183},
  year={2006},
  publisher={SIAM}
}

@inproceedings{AV21,
  title={{A Refined Laser Method and Faster Matrix Multiplication}},
  author={Alman, Josh and Williams, Virginia Vassilevska},
  booktitle={Proceedings of the 2021 ACM-SIAM Symposium on Discrete Algorithms (SODA)},
  pages={522--539},
  year={2021},
  organization={SIAM}
}

@article{Str69,
  title={Gaussian elimination is not optimal},
  author={Strassen, Volker},
  journal={Numerische mathematik},
  volume={13},
  number={4},
  pages={354--356},
  year={1969}
}

@inproceedings{CW87,
  title={Matrix multiplication via arithmetic progressions},
  author={Coppersmith, Don and Winograd, Shmuel},
  booktitle={Proceedings of the $19^{th}$ annual ACM symposium on Theory of computing},
  pages={1--6},
  year={1987}
}

@inproceedings{Gal14,
  title={Powers of tensors and fast matrix multiplication},
  author={Le Gall, Fran{\c{c}}ois},
  booktitle={Proceedings of the $39^{th}$ international symposium on symbolic and algebraic computation},
  pages={296--303},
  year={2014}
}

@article{DMMS11,
  title={{Faster Least Squares Approximation}},
  author={Drineas, Petros and Mahoney, Michael W. and Muthukrishnan, Shan and Sarl{\'o}s, Tam{\'a}s},
  journal={Numerische mathematik},
  volume={117},
  number={2},
  pages={219--249},
  year={2011},
  publisher={Springer}
}

@article{MMY15,
  title={A statistical perspective on algorithmic leveraging},
  author={Ma, Ping and Mahoney, Michael W. and Yu, Bin},
  journal={The Journal of Machine Learning Research},
  volume={16},
  number={1},
  pages={861--911},
  year={2015},
  publisher={JMLR. org}
}

@article{Rec11,
  title={{A Simpler Approach to Matrix Completion}},
  author={Recht, Benjamin},
  journal={Journal of Machine Learning Research},
  volume={12},
  number={12},
  year={2011}
}

@article{Mag10,
  title={{Row Sampling for Matrix Algorithms via a Non-Commutative Bernstein Bound}},
  author={Magdon-Ismail, Malik},
  journal={arXiv preprint arXiv:1008.0587},
  year={2010},
  publisher={Citeseer}
}

@inproceedings{AC06,
  title={{Approximate nearest neighbors and the fast Johnson-Lindenstrauss transform}},
  author={Ailon, Nir and Chazelle, Bernard},
  booktitle={Proceedings of the $38^{th}$ annual ACM symposium on Theory of computing},
  pages={557--563},
  year={2006}
}

@article{NL21,
  title = {{Optimal Sampling Algorithms for Block Matrix Multiplication}},
  journal = {Journal of Computational and Applied Mathematics},
  volume = {425},
  pages = {115063},
  year = {2023},
  issn = {0377-0427},
  doi = {https://doi.org/10.1016/j.cam.2023.115063},
  url = {https://www.sciencedirect.com/science/article/pii/S0377042723000079},
  author = {Chengmei Niu and Hanyu Li}
}

@inproceedings{Sar06,
  title={{Improved Approximation Algorithms for Large Matrices via Random Projections}},
  author={Sarl{\'o}s, Tam{\'a}s},
  booktitle={2006 $47^{th}$ Annual IEEE Symposium on Foundations of Computer Science (FOCS'06)},
  pages={143--152},
  year={2006},
  organization={IEEE}
}

@article{Oli10,
  title={{Sums of random Hermitian matrices and an inequality by Rudelson}},
  author={Oliveira, Roberto},
  journal={Electronic Communications in Probability},
  volume={15},
  pages={203--212},
  year={2010},
  publisher={Institute of Mathematical Statistics and Bernoulli Society}
}

@inproceedings{Spi10,
  title={{Algorithms, graph theory, and linear equations in Laplacian matrices}},
  author={Spielman, Daniel A.},
  booktitle={Proceedings of the International Congress of Mathematicians 2010 (ICM 2010) (In 4 Volumes) Vol. I: Plenary Lectures and Ceremonies Vols. II--IV: Invited Lectures},
  pages={2698--2722},
  year={2010},
  organization={World Scientific}
}

@inproceedings{CNW16,
  title={{Optimal Approximate Matrix Product in Terms of Stable Rank}},
  author={Cohen, Michael B. and Nelson, Jelani and Woodruff, David P.},
  booktitle={$43^{rd}$ International Colloquium on Automata, Languages, and Programming (ICALP 2016)},
  year={2016},
  organization={Schloss Dagstuhl-Leibniz-Zentrum fuer Informatik}
}

@inproceedings{CH23,
  title={{Graph Sparsification by Approximate Matrix Multiplication}},
  author={Charalambides, Neophytos and Hero, Alfred},
  booktitle={2023 IEEE Statistical Signal Processing Workshop (SSP)},
  pages={180-184},
  year={2023},
  organization={IEEE}
}

@article{HH21,
  author = {Hamm, Keaton and Huang, Longxiu},
  title = {{Perturbations of $CUR$ Decompositions}},
  journal = {SIAM Journal on Matrix Analysis and Applications},
  volume = {42},
  number = {1},
  pages = {351-375},
  year = {2021},
  doi = {10.1137/19M128394X},
  URL = {https://doi.org/10.1137/19M128394X},
  eprint = {https://doi.org/10.1137/19M128394X}
}

@INPROCEEDINGS{CMPH22,
  author={Charalambides, Neophytos and Mahdavifar, Hessam and Pilanci, Mert and Hero, Alfred},
  booktitle={2022 IEEE International Symposium on Information Theory (ISIT)}, 
  title={{Orthonormal Sketches for Secure Coded Regression}}, 
  year={2022},
  volume={},
  number={},
  pages={826-831},
  doi={10.1109/ISIT50566.2022.9834356}
}

@inproceedings{ZWL08,
 author = {Zhou, Shuheng and Wasserman, Larry and Lafferty, John},
 booktitle = {Advances in Neural Information Processing Systems},
 title = {{Compressed Regression}},
 volume = {20},
 year = {2008}
}

@inproceedings{DWZ23,
  title={{Faster Matrix Multiplication via Asymmetric Hashing}},
  author={Ran Duan and Hongxun Wu and Renfei Zhou},
  booktitle={2023 $64^{th}$ Annual IEEE Symposium on Foundations of Computer Science (FOCS'23)},
  pages={to appear},
  year={2023},
  organization={IEEE}
}

@article{WXXZ23,
  title={{New Bounds for Matrix Multiplication: from Alpha to Omega}},
  author={Virginia Vassilevska Williams and Yinzhan Xu and Zixuan Xu and Renfei Zhou},
  journal={arXiv preprint arXiv:2307.07970},
  year={2023}
}

@article{CMPH23,
  title={{Iterative Sketching for Secure Coded Regression}}, 
  author={Neophytos Charalambides and Hessam Mahdavifar and Mert Pilanci and Alfred Hero},
  journal={arXiv preprint arXiv:2308.04185},
  year={2023}
}

@article{LWMRT07,
  title={Randomized algorithms for the low-rank approximation of matrices},
  author={Liberty, Edo and Woolfe, Franco and Martinsson, Per-Gunnar and Rokhlin, Vladimir and Tygert, Mark},
  journal={Proceedings of the National Academy of Sciences},
  volume={104},
  number={51},
  pages={20167--20172},
  year={2007},
  publisher={National Acad Sciences}
}

@inproceedings{Pag12,
  author = {Pagh, Rasmus},
  title = {{Compressed Matrix Multiplication}},
  year = {2012},
  isbn = {9781450311151},
  publisher = {Association for Computing Machinery},
  address = {New York, NY, USA},
  url = {https://doi.org/10.1145/2090236.2090271},
  doi = {10.1145/2090236.2090271},
  booktitle = {Proceedings of the $3^{rd}$ Innovations in Theoretical Computer Science Conference},
  pages = {442–451},
  numpages = {10},
  location = {Cambridge, Massachusetts},
  series = {ITCS '12}
}

@inproceedings{CLMMPS15,
  author = {Cohen, Michael B. and Lee, Yin Tat and Musco, Cameron and Musco, Christopher and Peng, Richard and Sidford, Aaron},
  title = {{Uniform Sampling for Matrix Approximation}},
  year = {2015},
  isbn = {9781450333337},
  publisher = {Association for Computing Machinery},
  address = {New York, NY, USA},
  url = {https://doi.org/10.1145/2688073.2688113},
  doi = {10.1145/2688073.2688113},
  booktitle = {Proceedings of the 2015 Conference on Innovations in Theoretical Computer Science},
  pages = {181–190},
  numpages = {10},
  location = {Rehovot, Israel},
  series = {ITCS '15}
}

@inproceedings{KK15,
  author = {Kogan, Dmitry and Krauthgamer, Robert},
  title = {{Sketching Cuts in Graphs and Hypergraphs}},
  year = {2015},
  isbn = {9781450333337},
  publisher = {Association for Computing Machinery},
  address = {New York, NY, USA},
  url = {https://doi.org/10.1145/2688073.2688093},
  doi = {10.1145/2688073.2688093},
  booktitle = {Proceedings of the 2015 Conference on Innovations in Theoretical Computer Science},
  pages = {367–376},
  numpages = {10},
  location = {Rehovot, Israel},
  series = {ITCS '15}
}

@inproceedings{ACKQWZ16,
  author = {Andoni, Alexandr and Chen, Jiecao and Krauthgamer, Robert and Qin, Bo and Woodruff, David P. and Zhang, Qin},
  title = {{On Sketching Quadratic Forms}},
  year = {2016},
  isbn = {9781450340571},
  publisher = {Association for Computing Machinery},
  address = {New York, NY, USA},
  url = {https://doi.org/10.1145/2840728.2840753},
  doi = {10.1145/2840728.2840753},
  booktitle = {Proceedings of the 2016 ACM Conference on Innovations in Theoretical Computer Science},
  pages = {311–319},
  numpages = {9},
  location = {Cambridge, Massachusetts, USA},
  series = {ITCS '16}
}

@InProceedings{AW18,
  author =	{Josh Alman and Virginia Vassilevska Williams},
  title =	{{Further Limitations of the Known Approaches for Matrix Multiplication}},
  booktitle =	{$9^{th}$ Innovations in Theoretical Computer Science Conference (ITCS 2018)},
  pages =	{25:1--25:15},
  series =	{Leibniz International Proceedings in Informatics (LIPIcs)},
  ISBN =	{978-3-95977-060-6},
  ISSN =	{1868-8969},
  year =	{2018},
  volume =	{94},
  editor =	{Anna R. Karlin},
  publisher =	{Schloss Dagstuhl--Leibniz-Zentrum fuer Informatik},
  address =	{Dagstuhl, Germany},
  URL =		{http://drops.dagstuhl.de/opus/volltexte/2018/8360},
  URN =		{urn:nbn:de:0030-drops-83609},
  doi =		{10.4230/LIPIcs.ITCS.2018.25}
}

@InProceedings{MNSUW18,
  author =	{Cameron Musco and Praneeth Netrapalli and Aaron Sidford and Shashanka Ubaru and David P. Woodruff},
  title =	{{Spectrum Approximation Beyond Fast Matrix Multiplication: Algorithms and Hardness}},
  booktitle =	{$9^{th}$ Innovations in Theoretical Computer Science Conference (ITCS 2018)},
  pages =	{8:1--8:21},
  series =	{Leibniz International Proceedings in Informatics (LIPIcs)},
  ISBN =	{978-3-95977-060-6},
  ISSN =	{1868-8969},
  year =	{2018},
  volume =	{94},
  editor =	{Anna R. Karlin},
  publisher =	{Schloss Dagstuhl--Leibniz-Zentrum fuer Informatik},
  address =	{Dagstuhl, Germany},
  URL =		{http://drops.dagstuhl.de/opus/volltexte/2018/8339},
  URN =		{urn:nbn:de:0030-drops-83397},
  doi =		{10.4230/LIPIcs.ITCS.2018.8}
}

@inproceedings{AALG18,
  author       = {Vedat Levi Alev and  Nima Anari and Lap Chi Lau and Shayan Oveis Gharan},
  editor       = {Anna R. Karlin},
  title        = {{Graph Clustering using Effective Resistance}},
  booktitle    = {$9^{th}$ Innovations in Theoretical Computer Science Conference, {ITCS} 2018, January 11-14, 2018, Cambridge, MA, {USA}},
  series       = {LIPIcs},
  volume       = {94},
  pages        = {41:1--41:16},
  publisher    = {Schloss Dagstuhl - Leibniz-Zentrum f{\"{u}}r Informatik},
  year         = {2018},
  url          = {https://doi.org/10.4230/LIPIcs.ITCS.2018.41},
  doi          = {10.4230/LIPIcs.ITCS.2018.41},
  bibsource    = {dblp computer science bibliography, https://dblp.org}
}

@InProceedings{AKK19,
  author =	{Sepehr Assadi and Michael Kapralov and Sanjeev Khanna},
  title =	{{A Simple Sublinear-Time Algorithm for Counting Arbitrary Subgraphs via Edge Sampling}},
  booktitle =	{$10^{th}$ Innovations in Theoretical Computer Science  Conference (ITCS 2019)},
  pages =	{6:1--6:20},
  series =	{Leibniz International Proceedings in Informatics (LIPIcs)},
  ISBN =	{978-3-95977-095-8},
  ISSN =	{1868-8969},
  year =	{2018},
  volume =	{124},
  editor =	{Avrim Blum},
  publisher =	{Schloss Dagstuhl--Leibniz-Zentrum fuer Informatik},
  address =	{Dagstuhl, Germany},
  URL =		{http://drops.dagstuhl.de/opus/volltexte/2018/10099},
  URN =		{urn:nbn:de:0030-drops-100996},
  doi =		{10.4230/LIPIcs.ITCS.2019.6}
}

@InProceedings{HLMM20,
  author =	{Yael Hitron and Nancy Lynch and Cameron Musco and Merav Parter},
  title =	{{Random Sketching, Clustering, and Short-Term Memory in Spiking Neural Networks}},
  booktitle =	{$11^{th}$ Innovations in Theoretical Computer Science Conference (ITCS 2020)},
  pages =	{23:1--23:31},
  series =	{Leibniz International Proceedings in Informatics (LIPIcs)},
  ISBN =	{978-3-95977-134-4},
  ISSN =	{1868-8969},
  year =	{2020},
  volume =	{151},
  editor =	{Thomas Vidick},
  publisher =	{Schloss Dagstuhl--Leibniz-Zentrum fuer Informatik},
  address =	{Dagstuhl, Germany},
  URL =		{https://drops.dagstuhl.de/opus/volltexte/2020/11708},
  URN =		{urn:nbn:de:0030-drops-117087},
  doi =		{10.4230/LIPIcs.ITCS.2020.23}
}

@InProceedings{GGMW20,
  author =	{Shafi Goldwasser and Ofer Grossman and Sidhanth Mohanty and David P. Woodruff},
  title =	{{Pseudo-Deterministic Streaming}},
  booktitle =	{$11^{th}$ Innovations in Theoretical Computer Science Conference (ITCS 2020)},
  pages =	{79:1--79:25},
  series =	{Leibniz International Proceedings in Informatics (LIPIcs)},
  ISBN =	{978-3-95977-134-4},
  ISSN =	{1868-8969},
  year =	{2020},
  volume =	{151},
  editor =	{Thomas Vidick},
  publisher =	{Schloss Dagstuhl--Leibniz-Zentrum fuer Informatik},
  address =	{Dagstuhl, Germany},
  URL =		{https://drops.dagstuhl.de/opus/volltexte/2020/11764},
  URN =		{urn:nbn:de:0030-drops-117644},
  doi =		{10.4230/LIPIcs.ITCS.2020.79}
}

@InProceedings{BCGPU23,
  author =	{Blasiak, Jonah and Cohn, Henry and Grochow, Joshua A. and Pratt, Kevin and Umans, Chris},
  title =	{{Matrix Multiplication via Matrix Groups}},
  booktitle =	{$14^{th}$ Innovations in Theoretical Computer Science Conference (ITCS 2023)},
  pages =	{19:1--19:16},
  series =	{Leibniz International Proceedings in Informatics (LIPIcs)},
  ISBN =	{978-3-95977-263-1},
  ISSN =	{1868-8969},
  year =	{2023},
  volume =	{251},
  editor =	{Tauman Kalai, Yael},
  publisher =	{Schloss Dagstuhl -- Leibniz-Zentrum f{\"u}r Informatik},
  address =	{Dagstuhl, Germany},
  URL =		{https://drops.dagstuhl.de/opus/volltexte/2023/17522},
  URN =		{urn:nbn:de:0030-drops-175226},
  doi =		{10.4230/LIPIcs.ITCS.2023.19}
}

@inproceedings{Raz02,
  title={{On the Complexity of Matrix Product}},
  author={Raz, Ran},
  booktitle={Proceedings of the $34^{th}$ annual ACM symposium on Theory of computing},
  pages={144--151},
  year={2002}
}

@phdthesis{Sto10,
  title={{On the Complexity of Matrix Multiplication}},
  author={Vioreanu, Andrew},
  year={2012},
  school={Yale University}
}

@article{CW90,
  title = {{Matrix Multiplication via Arithmetic Progressions}},
  journal = {Journal of Symbolic Computation},
  volume = {9},
  number = {3},
  pages = {251-280},
  year = {1990},
  note = {Computational algebraic complexity editorial},
  issn = {0747-7171},
  doi = {https://doi.org/10.1016/S0747-7171(08)80013-2},
  url = {https://www.sciencedirect.com/science/article/pii/S0747717108800132},
  author = {Don Coppersmith and Shmuel Winograd},
}

@inproceedings{YM23,
  title={{Efficient Block Approximate Matrix Multiplication}},
  author={Yang, Chuhan and Musco, Christopher},
  booktitle={$31^{st}$ Annual European Symposium on Algorithms (ESA 2023)},
  year={2023},
  organization={Schloss Dagstuhl-Leibniz-Zentrum f{\"u}r Informatik}
}

@article{CJ21,
  author = {Chen, Ke and Jin, Ruhui},
  title = {Tensor-Structured Sketching for Constrained Least Squares},
  journal = {SIAM Journal on Matrix Analysis and Applications},
  volume = {42},
  number = {4},
  pages = {1703-1731},
  year = {2021},
  doi = {10.1137/20M1374596},
  URL = {https://doi.org/10.1137/20M1374596},
  eprint = {https://doi.org/10.1137/20M1374596}
}

@article{MS21,
  title={{Fast and Accurate Randomized Algorithms for Low-rank Tensor Decompositions}},
  author={Ma, Linjian and Solomonik, Edgar},
  journal={Advances in Neural Information Processing Systems},
  volume={34},
  pages={24299--24312},
  year={2021}
}

@article{Wu19,
  title={{A Note on Random Sampling for Matrix Multiplication}},
  author={Wu, Yue},
  journal={arXiv preprint arXiv:1811.11237},
  year={2019}
}

@article{MD09,
  title={$CUR$ matrix decompositions for improved data analysis},
  author={Mahoney, Michael W. and Drineas, Petros},
  journal={Proceedings of the National Academy of Sciences},
  volume={106},
  number={3},
  pages={697--702},
  year={2009},
  publisher={National Acad Sciences}
}

@INPROCEEDINGS{KVZ14,
  author={Kyrillidis, Anastasios and Vlachos, Michail and Zouzias, Anastasios},
  booktitle={2014 IEEE International Symposium on Information Theory}, 
  title={{Approximate Matrix Multiplication with Application to Linear Embeddings}}, 
  year={2014},
  pages={2182-2186}
}

@inproceedings{MMG16,
  title={Co-occurring directions sketching for approximate matrix multiply},
  author={Mroueh, Youssef and Marcheret, Etienne and Goel, Vaibahava},
  booktitle={Artificial Intelligence and Statistics},
  pages={567--575},
  year={2017},
  organization={PMLR}
}

@article{NT14,
  title={{Paved with good intentions: analysis of a randomized block Kaczmarz method}},
  author={Needell, Deanna and Tropp, Joel A.},
  journal={Linear Algebra and its Applications},
  volume={441},
  pages={199--221},
  year={2014},
  publisher={Elsevier}
}

@article{DM08,
  title={{Relative-error $CUR$ matrix decompositions}},
  author={Drineas, Petros and Mahoney, Michael W.},
  journal={SIAM Journal on Matrix Analysis and Applications},
  volume={30},
  number={2},
  pages={844--881},
  year={2008},
  publisher={SIAM}
}

@article{CYD19,
title = {Structural conditions for projection-cost preservation via randomized matrix multiplication},
journal = {Linear Algebra and its Applications},
volume = {573},
pages = {144-165},
year = {2019},
issn = {0024-3795},
doi = {https://doi.org/10.1016/j.laa.2019.03.013},
url = {https://www.sciencedirect.com/science/article/pii/S0024379519301144},
author = {Agniva Chowdhury and Jiasen Yang and Petros Drineas}
}

@article{LR24,
title = {A subspace constrained randomized Kaczmarz method for structure or external knowledge exploitation},
journal = {Linear Algebra and its Applications},
volume = {698},
pages = {220-260},
year = {2024},
issn = {0024-3795},
doi = {https://doi.org/10.1016/j.laa.2024.06.010},
url = {https://www.sciencedirect.com/science/article/pii/S0024379524002647},
author = {Jackie Lok and Elizaveta Rebrova}
}

@article{BZMD14,
  title={{Randomized Dimensionality Reduction for $k$-means Clustering}},
  author={Boutsidis, Christos and Zouzias, Anastasios and Mahoney, Michael W. and Drineas, Petros},
  journal={IEEE Transactions on Information Theory},
  volume={61},
  number={2},
  pages={1045--1062},
  year={2014},
  publisher={IEEE}
}

@inproceedings{CEMMP15,
  title={{Dimensionality Reduction for $k$-Means Clustering and Low Rank Approximation}},
  author={Cohen, Michael B. and Elder, Sam and Musco, Cameron and Musco, Christopher and Persu, Madalina},
  booktitle={Proceedings of the forty-seventh annual ACM symposium on Theory of computing},
  pages={163--172},
  year={2015}
}

@article{Cha20,
  title={{Dimensionality Reduction for $k$-means Clustering}},
  author={Charalambides, Neophytos},
  journal={arXiv preprint arXiv:2007.13185},
  year={2020}
}

@inproceedings{DMM06,
  title={{Sampling Algorithms for $\ell_2$ Regression and Applications}},
  author={Drineas, Petros and Mahoney, Michael W. and Muthukrishnan, Shan},
  booktitle={Proceedings of the seventeenth annual ACM-SIAM symposium on Discrete algorithm},
  pages={1127--1136},
  year={2006}
}

@phdthesis{Neo23,
  title={{Coding Theory and Randomized Sketching for Distributed Optimization}},
  author={Charalambides, Neophytos},
  year={2023},
  school={University of Michigan}
}

@article{CPH23b,
  author={Charalambides, Neophytos and Pilanci, Mert and Hero, Alfred O.},
  journal={IEEE Transactions on Information Theory}, 
  title={{Gradient Coding With Iterative Block Leverage Score Sampling}}, 
  year={2024},
  volume={70},
  number={9},
  pages={6639-6664},
  doi={10.1109/TIT.2024.3420222}
}

@inproceedings{OJXE18,
  title={{Block $CUR$: Decomposing Matrices Using Groups of Columns}},
  author={Oswal, Urvashi and Jain, Swayambhoo and Xu, Kevin S and Eriksson, Brian},
  booktitle={Joint European Conference on Machine Learning and Knowledge Discovery in Databases},
  pages={360--376},
  year={2018},
  organization={Springer}
}

@article{NL23,
  title={{Optimal Sampling Algorithms for Block Matrix Multiplication}},
  author={Niu, Chengmei and Li, Hanyu},
  journal={Journal of Computational and Applied Mathematics},
  volume={425},
  pages={115063},
  year={2023},
  publisher={Elsevier}
}


\end{document}